\documentclass[preprint,11pt]{elsarticle}
\usepackage{amsmath}
\usepackage{amssymb}
\usepackage{datetime}
\usepackage{amsthm}
\usepackage{mathrsfs}
  \usepackage[margin=3cm]{geometry}
\usepackage{lipsum}
\usepackage{graphicx}
\usepackage[usenames, dvipsnames]{xcolor}
\def\bel{\begin{equation}\label}
\def\eeq{\end{equation}} 
\def\bel{\begin{equation}\label}
\def\eeq{\end{equation}}
\newtheorem{definition}{Definition}[section]
\newtheorem{theorem}{Theorem}[section]
\newtheorem{remark}{Remark}[section]
\newtheorem{lemma}{Lemma}[section]

\newtheorem{proposition}{Proposition}[section]

\newtheorem{example}{Example}[section]

\def\bel{\begin{equation}\label}
\def\eeq{\end{equation}}
\journal{ }
\def\fudge{\mathchoice{}{}{\mkern.5mu}{\mkern.8mu}}
\def\bbc#1#2{{\rm \mkern#2mu\vbar\mkern-#2mu#1}}
\def\bbb#1{{\rm I\mkern-3.5mu #1}}
\def\bba#1#2{{\rm #1\mkern-#2mu\fudge #1}}
\def\bb#1{{\count4=`#1 \advance\count4by-64 \ifcase\count4\or\bba A{11.5}\or0
\bbb B\or\bbc C{5}\or\bbb D\or\bbb E\or\bbb F \or\bbc G{5}\or\bbb H\or
\bbb I\or\bbc J{3}\or\bbb K\or\bbb L \or\bbb M\or\bbb N\or\bbc O{5} \or
\bbb P\or\bbc Q{5}\or\bbb R\or\bbc S{4.2}\or\bba T{10.5}\or\bbc U{5}\or%
\bbb P\or\bbc Q{5}\or\bbb R\or\bba S{8}\or\bba T{10.5}\or\bbc U{5}\or
\bba V{12}\or\bba W{16.5}\or\bba X{11}\or\bba Y{11.7}\or\bba Z{7.5}\fi}}
\usepackage{colonequals}

\def \B {\mathbb{B}}
\def \R {\mathbb{R}}

\def \C {{\mathcal  C}}

\def \N {{\bb N}}

\def \vsm{\vskip 0.2 truecm}
\def \vv{\vskip 0.5 truecm}

\def\bel{\begin{equation}\label}
\def\eeq{\end{equation}}

\usepackage{color}

\def \vsm{\vskip 0.2 truecm}
\def \ds{\displaystyle}
\def \vv{\vskip 0.5 truecm}

\def \T{{\mathcal T}}
\def \CC{{\bold C}}
\def \C{\mathcal C}

\def\ds{\displaystyle}

\def\bega{\begin{array}}
\def\enda{\end{array}}

\begin{document}
\begin{frontmatter}
\title{Normality and Gap Phenomena  in Optimal Unbounded    Control\tnoteref{label1}}
\tnotetext[label1]{This research is partially supported by the  INdAM-GNAMPA Project 2017 "Optimal impulsive control: higher order necessary conditions and gap phenomena";  and by the Padua University grant PRAT 2015 ``Control
of dynamics with reactive constraints''}
\author{Monica Motta}
  \ead{motta@math.unipd.it}
\cortext[cor1]{M. Motta}
\address{Dipartimento di Matematica,
Universit\`a di Padova\\ Via Trieste, 63, Padova  35121, Italy\\
Telefax (39)(49) 827 1499,\,\, Telephone (39)(49) 827 1368} 
 \author{Franco Rampazzo}
  \ead{rampazzo@math.unipd.it}
\address{Dipartimento di Matematica,
Universit\`a di Padova\\ Via Trieste, 63, Padova  35121, Italy\\
Telefax (39)(49) 827 1499,\,\, Telephone (39)(49) 827 1342} 
\author{Richard Vinter}
  \ead{r.vinter@imperial.ac.uk}
\address{Department of Electrical and Electronic Engineering, \\ Imperial College London, Exhibition Road, London SW7 2BT, UK\\
 Telefax (+44) 207 594 6282,\,\, Telephone (+44) 207 594 6287}  
\date{\today}

\begin{abstract} 
 Optimal unbounded control problems with linear growth w.r.t. the control, both in the dynamics and in the cost, may fail to have  minimizers in the class of absolutely continuous state trajectories. For this reason, extended versions of such problems have been investigated, in which the domain is extended to include possibly discontinuous state trajectories of bounded variation, and for which existence of minimizers is guaranteed. It is of interest to know whether the passage from the original optimal control problem to its extension introduces an infimum gap. This will reveal whether it is possible to approximate extended minimizers by absolutely continuous state trajectories, as might be required for engineering implementation, and whether numerical schemes might be ill-conditioned. This paper provides  sufficient conditions under which there is no infimum gap, expressed in terms of normality of extremals. The link we establish between infimum gaps and normality gives insights into the infimum gap phenomenon. But, perhaps more importantly, it opens up a new approach to devising useful tests for the absence of infimum gaps, namely to supply verifiable sufficient conditions for normality of extremals. We give several examples of the use of this approach, and show that it leads to either new conditions, or improvement of known conditions, for no infimum gaps. We also give a criterion for non infimum gaps, which covers some problems where the normality condition is violated, illustrating that sufficient conditions of normality type, while covering many cases, are not necessary. 
\end{abstract}

\end{frontmatter}

\section{Introduction}
 It is well known in Optimization Theory  that the infimum cost is `stable' under structural changes, if the optimization problem considered is {\it normal}, i.e. Lagrange multiplier rules are valid in a form that requires the cost multiplier to differ from zero (see, e.g., \cite{bonnans}, \cite{Dontchev}).  Likewise  in  Optimal  Control Theory, it has been shown that, under normality-type hypotheses, the infimum cost will not decrease  when we enlarge the class of state trajectories to include relaxed state trajectories, i.e. trajectories whose derivatives lie  in the convexified velocity set, a procedure that can be interpreted as a (infinite-dimensional) structural change to the domain of the optimal control problem under consideration.   When there is no such decrease, we say  `there is no infimum gap' (see \cite{warga}, \cite{warga1}, \cite{PV1},  \cite{PV2}). 
 We remark that no infimum gap can occur when the right endpoint of state trajectories is free, in consequence of the Relaxation Theorem  which tells us that  the set  of state trajectories is `$C^0$-dense' in the set of  relaxed state trajectories. But infimum gaps may arise when the right endpoint is constrained.
\ \\

 Similar considerations come into play in connection with the following class of optimal control problems,  in which
the original dynamics {are} unbounded and the customary  coercivity hypotheses, which have the effect of excluding optimal trajectories that are discontinuous, are no longer  invoked.

\begin{equation}\label{intro1}  
(P)
\left\{
\begin{array}{l}
\;\;\;\;\;\;\;\;\;\;\;\;\mbox{Minimize } h(t_1,x(t_1),t_2,x(t_2),v(t_2))
\\ [1.5ex]
\ds\mbox{over } t_1, \, t_2\in\R, \ t_1<t_2, \  (x,v, u) \in W^{1,1}([t_1,t_2];\R^{n}\times\R\times\R^m)  
\ \mbox{satisfying }  
\\ [1.5ex]
\displaystyle\frac{dx}{dt}(t)\,=\, f(t,x(t)) + \sum_{j=1}^{m}g_{j}(t,x(t))\frac{du^j}{dt}(t) \, \quad \mbox{ a.e. } t \in [t_1,t_2]  \\ [1.5ex]
\displaystyle\frac{dv}{dt}(t)\,=\,  \left|\frac{du}{dt}(t)\right|  \, \quad \mbox{ a.e. } t \in [t_1,t_2]  \\ [1.5ex]
 \displaystyle\frac{du}{dt}(t) \in \mathcal{C}\,\, \mbox{ a.e. } t \in [t_1,t_2], \\ [1.5ex]
\ds v(t_1)=0, \quad v(t_2)\leq K,     \quad  \big(t_1,x(t_1),t_2, x(t_2)\big) \in {\mathcal T},
\end{array}
\right.
\end{equation}
where  $K>0$ is a fixed constant (possibly equal to $+\infty$), $\mathcal{C}\subseteq\R^m$ is  a closed convex cone,  and the end-point constraint ${\mathcal T}\subseteq
 \R\times\R^n\times\R\times\R^n$ is a closed subset.  Notice that for every $t\in[t_1,t_2]$,  $v(t)$ is nothing but  the {\em total variation of $u $  on $[t_1,t]$. } In particular,  the choice  $K=+\infty$  means  that  there are no  constraints on the total variation of the controls $u $. 
 Since the  derivatives $\ds \frac{du}{dt}$  -- here playing the role of controls -- are not $L^{\infty}$ uniformly bounded, in general, due to a lack of  coerciveness, minima do not exist in the set of $W^{1,1}$ trajectories, even if one assumes the control cone $\mathcal C$ to be convex. So, in general, minimizing sequences  of trajectories do not even converge  to a continuous path. This is why we call these problems {\it impulsive}. We remark that the occurrence of impulses (i.e. discontinuities of $(x,u,v)(\cdot)$) is not ruled out by the weak coerciveness assumption $K < +\infty$.  However, through a time-reparameterization $t=t(s)$ based on the   total variation  map  $v(\cdot)$  one can  reformulate  problem (P) as  the following equivalent  space-time minimum control problem, having  the  same form but with   {\it bounded} control derivatives.
 
\begin{equation}\label{introextended}
(P_e) 
\left\{
\begin{array}{l}
\displaystyle
\mbox{Minimize } h(y^0(0),y(0),y^0(S), y(S),\nu(S))
\\ [1.5ex]
 \displaystyle\mbox{over } S> 0,\, ({y^0}, y,\nu, \varphi^0,\varphi)  \in W^{1,1}([0,S];\R\times\R^{n}\times\R\times\R\times \R^m)\,\,\,
\mbox{satisfying }
\\ [1.5ex]
\begin{array}{l}
\displaystyle\frac{d{y^0}}{ds} (s)\,=\, \frac{d\varphi^0}{ds}(s) \ \mbox{ a.e. } s \in [0,S]\,, \\ [1.5ex]
\displaystyle\frac{dy}{ds} (s)\,=\, f(y^0(s), y(s))\frac{d\varphi^0}{ds}(s)+ \sum_{j=1}^{m}g_{j}(y^0(s), y(s))\frac{d\varphi^j}{ds}(s) \ \mbox{ a.e. } s \in [0, S]\,,
\\ [1.5ex]
\displaystyle\frac{d{\nu}}{ds} (s)\,=\, \left|\frac{d\varphi}{ds}(s)\right| \ \mbox{ a.e. } s \in [0,S]\,, \\ [1.5ex]
   \displaystyle
 \bigg(\frac{d\varphi^0}{ds},\frac{d\varphi}{ds}\bigg)(s)\in \CC 
\  \mbox{ a.e. } s \in [0,S]\,, 
\\ [1.5ex]
\displaystyle \nu(0)=0, \quad \nu(S)\le K,\,\,\quad\Big(y^0(0),y(0) , y^0(S),y(S)\Big)\in \T,  \end{array}
\end{array}
\right. \end{equation}
where
$$
\CC:=\left\{(w^0,w)\in \R_+\times \C  :  \ w^0+|w | = 1\right\}
$$
and  $\varphi^0: [0,S]\to [t_1,t_2]$ is a   surjective, {\it non-decreasing}, {parameterization of  $[t_1,t_2]$}. 
The controls  $(\varphi^0,\varphi)$    verify  $\frac{d\varphi^0}{ds}(s)+\left|\frac{d\varphi}{ds}(s)\right|=1$ a.e., so    they  are $1$-Lipschitz continuous. The embedding of the  problem \eqref{intro1}  into the standard control problem \eqref{introextended} is obtained  by setting  
\begin{equation}\label{introembed} s(t):=\int_{t_1}^t \left(1+\left|\frac{du}{d\tau}(\tau)\right| \right) d\tau \quad (\varphi^0,\varphi) := (id, u)\circ s^{-1}  \quad (y^0,y) := (id, x)\circ s^{-1} .
\end{equation} 
Let us point out that by allowing the time-maps    $\varphi^0$ {\it to be constant on non-degenerate intervals} \footnote{A distributional, or measure-theoretical, approach is out of question here, in view of the fact that the vector fields $g_i$ are allowed to be non-commutative, see  \cite{BR} \cite{Haiek}.},  we arrive at parameterized limits of   trajectory  graphs  --here called {\it extended sense trajectories}-- which are not graphs of trajectories  of the original problem. For a  selective  bibliography on this subject, we refer the reader to \cite{BR},   \cite{MR}, \cite{MR1}, \cite{MR2}, \cite{MS},  \cite{MiRu},  \cite{SV}, \cite{KDPS},  \cite{GS}, \cite{Suss}, \cite{WZ} \cite{AKP1},   \cite{AKP3},     \cite{SV1}, \cite{MiRu}  and to the references therein.
  In particular,  the state trajectories $y $ may happen to be non-constant on $s$-subintervals where the  real time $ t=y^0 $ is constant. On the other hand,  it is known that  $s\mapsto(\varphi^0,\varphi)(s) $ coincides with the reparameterization $s\mapsto (id,u)(t(s))$ of the  graph $t\mapsto (id,u)(t)$ of some {\it strict sense}--i.e. absolutely continuous--  control  $u$ if and only if $\ds\frac{d\varphi^0}{ds}>0$ almost everywhere.
 
   Since the set of embedded strict sense trajectories \footnote{Namely, the extended sense trajectories associated to graph reparameterizations of strict-sense controls.}  is $C^0$-dense in the set of extended sense trajectories, it is natural to address the infimum gap issue, as has previously been done in the context of relaxation for optimal control problems  (\cite{warga}, \cite{warga1}, \cite{PV1},  \cite{PV2}). That is, one can   question whether  the infimum value of the functional among the strict-sense trajectories happens to be greater than the corresponding infimum among the extended-sense trajectories.
Such infimum {gaps} may actually occur, as shown by simple examples. In the present  paper we aim to   explore the  connection between `normality' of a minimizer of problem   $(P_{e})$ and the   occurrence of  an infimum gap. In this case,  {\it normality} means that if $(p_0 ,p ,\pi,\lambda)$ is {an arbitrary } set of `multipliers' ( i.e., an adjoint path)   appearing  in the Maximum Principle \footnote{The adjoint component $p_0$ corresponds to the time variable $y^0$, while $\pi$ corresponds to the variable `total variation', $\nu$ (see below).}, then $\lambda > 0$.  Omitting some details concerning the precise hypotheses we impose on the data, our main result (see Theorem \ref{cor}) may be summarized as follows: 
 \begin{theorem}
 Consider the optimal control problem \eqref{intro1} and its extended sense formulation \eqref{introextended}. Assume that
  some minimizer for  $(P_{e})$ is a normal extremal.   Then there is no infimum gap. 
  \end{theorem}   
 As we shall see, this relation  is proved by means of  a topological argument  concerning  the properties of so-called  {\it isolated} extended sense feasible trajectories (see Theorem \ref{mainextended}). 
 We  point out that  our normality hypotheses are of more theoretical than practical interest.  Indeed to check them   requires knowledge of all sets of Lagrange multipliers $(p_0 ,p ,\pi,\lambda)$ in the Maximum   Principle, associated with some minimizer.   Yet, this drawback may be overcome  in  certain situations  in which  verifiable condition on the data can be identified, guaranteeing that every set of multipliers is normal. Such  considerations are the motivation for  the last part of the paper (see Section \ref{SNG}),  where  several  sufficient conditions for normality are provided in terms of quite  reasonable assumptions on the end-point constraints and the vector fields $f,g_1,\dots,g_m$. 
 
 A `no gap condition' is clearly desirable, in particular when numerical schemes are employed to solve specific problems.  We point out that  consideration of  impulsive systems is crucial in many  applications  (see, e.g. \cite{AB},  \cite{CSWML}, \cite{GRR}).    Instances in mechanics are   situations where  some   state parameters $(u_1,\ldots,u_m)$ are regarded as controls (\cite{AB1}, \cite{AB2},    \cite{BR1}). The fact that the derivatives of  these controls  appear linearly in the dynamical equations (rather than  quadratically) is an intrinsic, metric property of the foliation $u=c, c\in \R^m$, when the space of states is  endowed with the kinetic energy metric. Examples where this  property is verified    include  a swing where the length  of the swing is  regarded as the control input,  and a multiple pendulum where  the  control  inputs  are identified with  the mutual angles between adjacent  single pendulums (see e.g. \cite{R}, \cite{BP}).

\vskip0.5truecm
The paper is organized as follows: in  Section 2 we describe how the original problem   is embedded in the extended problem; Section 3 concerns a Maximum Principle for the extended problem, which, compared with the standard version, includes an improved non-triviality condition.  In Section 4  we prove that an {\it  isolated}  extended-sense extremal cannot be normal and, as a corollary, we deduce the main result.  In Section 5  we identify some classes of  problems where normality --hence the non occurrence of gap-phenomena--  can be established a priori, without any knowledge of the multipliers associated with the given extended-sense minimizer, as earlier anticipated.  Finally, in   Section 6, we  present  some  instructive examples --including one where normality is shown to be not necessary to rule out gap phenomena.

\vskip0.5truecm
\noindent{\sc Preliminaries and notation.}  In an Euclidean space of dimension $N$, the norm $|x|$ of a vector $x$ is defined as $$  |x|:=\left(\sum_{i=1}^{N}(x^i)^2\right)^{1/2},$$  and the closed unit ball $\{x \,|\, |x| \leq 1\}$ is denoted by $\B_N$.  The surface of the unit ball $\{x\,|\, |x|=1\}$ is written $\partial \B_N$.   $\R^+:=[0,+\infty)$.  $d_{D}(x)$ denotes the Euclidean distance of the point $x\in E$ from a given closed subset  $D\subset E$, namely 
$d_{D}(x)\colonequals\min\{|x-x'|\,|\,x' \in D \}$. 
 
 \noindent
Some standard constructs from nonsmooth analysis are employed in this paper. Background material may be found in a number of texts, examples of which include \cite{ClarkeLed}, \cite{RW} and \cite{Vinter}.
 
\vsm\begin{definition}
Take a closed set $D\subset \R^{k}$ and a point $\bar{x} \in D$. The {\rm limiting normal cone} $N_D(\bar{x})$ of $D$ at $\bar{x}$ is defined to be
\begin{eqnarray*}
N_D(\bar{x})&:=& \Bigg\{ \;
p \; | \; \exists \; x_{i} \stackrel{D}{\longrightarrow}\bar{x}
, \; p_i \longrightarrow p \mbox{  s.t. } 
\limsup_{x \stackrel{D}{\rightarrow} x_{i}} \;\frac{p_i \cdot (x - x_i)}{|x - x_i|} \leq 0  \,\mbox{ for each } i 
\Bigg\},
\end{eqnarray*}
in which $x_{i} \stackrel{D}{\longrightarrow}\bar{x}$ is notation conveying the information `$x_{i}\rightarrow \bar{x}$' and `$x_{i}\in D$ for all $i$'. 
 \end{definition}

\begin{definition} Take  a lower semicontinuous function 
$f :\R^{k} \rightarrow \R$ and a point $\bar{x} \in \R^{k}$.
The {\it limiting subdifferential} of $f $ at $\bar{x}$ is
\begin{eqnarray*}
&&\partial f(\bar{x}) \;=\;  \Bigg\{ \xi\,|\, \exists\; \xi_{i} \rightarrow \xi \mbox{ and }
x_{i} \rightarrow \bar{x} \mbox{ s.t. }
 \limsup_{x \rightarrow  x_{i}} \frac{ \xi_{i} \cdot (x-x_{i})- f(x)+f(x_{i})}{|x-x_{i}|}
\, \leq  0 \, \mbox{ for each } i  \Bigg\}  . 
\end{eqnarray*}
\end{definition}
We take note, for future use, of the following facts about the limiting  subgradient of the distance function $d_{D} $ from an arbitrary closed  set $D$  (see, e.g., \cite[Section 4.8]{Vinter}):
 \vsm
{$i)$  if $\bar x \notin D$ and   $\xi\in\partial d_{D}(\bar x)$ then $|\xi|=1$ ;

$ii)$ $\partial d_{D}(\bar x)=N_D(\bar{x})\cap \B_k$   whenever  $\bar x \in D$.}

\section{The Optimal Control Problem and its Impulsive Extension}\label{SOC} 

\subsection{The optimal  control problem} 
Fix $K\in(0,+\infty]$,  a closed, convex cone  $\mathcal{C}\subseteq\R^m$ and a closed set ${\mathcal T}\subseteq
 \R\times\R^n\times\R\times\R^n$, and  consider the optimal control problem $(P)$  formulated in the Introduction.  We shall invoke the following hypothesis
\vsm

 \begin{itemize} \item[{\bf (H1)}: {\rm (i)}]
 the vector fields  $f :\R\times\R^{n} \rightarrow \R^{n}$, $g_{j} :\R\times\R^{n} \rightarrow \R^{n},\,j=1,\ldots,m$, are of class   $C^{1}$ and for some $A,B>0$,  verify
$$
|f(t,x)| + |g_1(t,x)| +\dots+|g_m(t,x)|\leq A|x| + B \quad \forall (t,x)\in\R\times\R^n;
$$
\item[{\rm (ii)}] the function $h :\R\times\R^{n}\times\R\times \R^{n}\times\R \rightarrow \R$ is of class $C^1$ and  for every $(t_1,x_1,t_2,x_2)$, the map $v\mapsto h(t_1,x_1,t_2,x_2,v)$ is monotone non-decreasing.
\end{itemize}

\begin{remark}\label{RemE}{\rm     
By means  of the  addition of  the trivial equations $\ds \frac{d\tilde x}{dt}(t) = \frac{du}{dt}(t)$,  where $\tilde x = (x^{n+1},\ldots, x^{n+m})$, we can allow $h$, $f$, $g_j$, $j=1,\dots,m$ to depend on $u$ as well as on $(t,x)$. 
}
\end{remark}
\vsm
\begin{definition}[Strict sense processes] 
Let  $t_1,t_2\in\R$ verify $t_1<t_2$. We   call a function $u \in W^{1,1}([t_1,t_2];\R^{m})$ a {\em strict sense control  on $[t_1,t_2]$}  if \, $ \ds \frac{du}{dt}(t) \in \mathcal{C}$ \,  a.e. in $[t_1,t_2]$. 
A {\em strict sense process} is a five-tuple  $(t_1,t_2, x ,v , u )$, $t_1<t_2$, in which $u $ is a strict sense control on $[t_1,t_2]$ and $(x,v) $ is a $W^{1,1}([t_1,t_2];\R^{n}\times\R)$ function satisfying 
\begin{equation}
\label{strict0}
\left\{
\begin{array}{l}
\displaystyle\frac{dx}{dt}(t)\,=\, f(t,x(t)) + \sum_{j=1}^{m}g_{j}(t,x(t))\frac{du^j}{dt}(t) \, \quad \mbox{ a.e. } t \in [t_1,t_2]  
\\
\displaystyle\frac{dv}{dt}(t)\,=\,  \left|\frac{du}{dt}(t)\right|  \, \quad \mbox{ a.e. } t \in [t_1,t_2]. 
 \end{array}
 \right.
 \end{equation} 
 If $(t_1,t_2, x ,v , u )$ is a strict sense process,  the four-tuple $(t_1,t_2,x ,v)\in \R^2\times W^{1,1}([t_1,t_2];\R^{n}\times\R) $ is called the (strict sense) {\em trajectory corresponding to $(t_1,t_2, x ,v , u )$}.
If  a strict sense process $(t_1,t_2, x ,v , u )$ also satisfies $v(t_1)=0$, $v(t_2)\leq K$ and  $\big(t_1,x(t_1),t_2, x(t_2)\big) \in \mathcal{T}$,  we say it  is {\em feasible}.
\end{definition}
Let us observe that  there is a trivial   one to one correspondence between trajectories $(t_1,t_2,x ,v)$ and the four-tuple one obtains by  extending continuously  $(x,v)$ to $\R$ in such a way that   both of them are constant outside the original domain $[t_1,t_2]$. 
 In order to define a metric on  strict sense trajectories  we  shall always consider their  extension  to $\R$.  Let us define the distance 
\begin{equation}\label{dinfty}
 d_{\infty}\Big((t_1,t_2,x ,v) , (\bar t_1,\bar t_2, \bar x ,\bar v)\Big) := | t_1 - \bar t_1 | + | t_2 - \bar t_2|+  \|(x,v)-(\bar x,\bar v)   \|_{L^{\infty}(\R)}.
  \end{equation}

\begin{definition}[Local and global strict sense  minimizers] We say a feasible strict sense process $(\bar t_1,\bar t_2,\bar x ,\bar v , \bar u )$ is a {\em strict sense $L^{\infty}$ local minimizer } if  there exists $\delta >0$ such that
\begin{equation}\label{min1}
h(\bar t_1,\bar x(\bar t_1), \bar t_2, \bar x(\bar t_2),\bar v(\bar t_2))\,\leq\, h(t_1, x(t_1),t_2,  x(t_2),v(t_2))
\end{equation}
for all feasible strict sense processes $(t_1,t_2, x ,v , u )$ verifying 
$$
d_{\infty}\Big((t_1,t_2,x ,v) , (\bar t_1,\bar t_2, \bar x ,\bar v)\Big)  \leq \delta,
$$
 If  relation \eqref{min1} is satisfied for {\em all}  strict sense feasible processes $(t_1,t_2, x ,v , u )$, we say that $({\bar t}_1,{\bar t}_2, \bar x ,\bar v , \bar u )$ is a {  {\em strict sense $L^{\infty}$ (global) minimizer}.}
\end{definition}

  \subsection{The extended system }

Let $(t_1,t_2, x , v , u )$ be a strict sense process.  If we reparameterize  time  in the graph-equation 
\begin{equation}\label{graph}
\left\{\begin{array}{l}\ds\frac{dx^0}{dt}(t) = 1 \\ [1.5ex]
 \displaystyle\frac{dx}{dt}(t)\,=\, f(  x^0(t), x(t)) + \sum_{j=1}^{m}g_{j}(x^0(t),x(t))\frac{du^j}{dt}(t)  \quad \text{a.e. } t\in[t_1,t_2], \\ [1.5ex]
 \displaystyle\frac{dv}{dt}(t)\,=\,  \left|\frac{du}{dt}(t)\right|  \, \quad \mbox{ a.e. } t \in [t_1,t_2], \\ [1.5ex] 
 x^0(t_1)=t_1\,,
\end{array}\right.
\end{equation}
through a bi-Lipschitz  increasing, surjective map $\varphi^0:[0,S]\to [t_1,t_2]$,  we obtain the equivalent differential system  on $[0,S]$
 \begin{equation}
 \label{extended}
 \left\{
\begin{array}{l}
\displaystyle\frac{d{y^0}}{ds} (s)\,=\, \frac{d\varphi^0}{ds}(s) \ \mbox{ a.e. } s \in [0,S]\,, \\ [1.5ex]
\displaystyle\frac{dy}{ds} (s)\,=\, f(y^0(s), y(s))\frac{d\varphi^0}{ds}(s)+ \sum_{j=1}^{m}g_{j}(y^0(s), y(s))\frac{d\varphi^j}{ds}(s) \ \mbox{ a.e. } s \in [0, S]\,,
\\ [1.5ex]
\displaystyle\frac{d\nu}{ds} (s)\,=\, \left|\frac{d\varphi}{ds}(s)\right|  \ \mbox{ a.e. } s \in [0,S]\,, 
\\ [1.5ex]
\displaystyle\left(\frac{d\varphi^0}{ds},\frac{d\varphi}{ds}\right)(s)\in \CC
\  \mbox{ a.e. } s \in [0,S]\,,
\end{array}
\right. \end{equation}
where
$$
\CC:=\left\{(w^0,w)\in \R_+\times \C  :  \ w^0+|w | = 1\right\}.
$$
 If, however,  we allow the  map $\varphi^0 $ merely to be   non-decreasing, we arrive at a new, {\it impulsive}, system, in which  the $s$-intervals where $\varphi^0 $ is constant represent the {\it arcs of (nonlinear) instantaneous evolution}   of both the control  and the state:

\begin{definition}[Extended Sense Processes] An {\em extended sense  process} $(S,{y^0} , y ,\nu, \varphi^0 ,\varphi )$ comprises $S\geq 0$ and  Lipschitz continuous functions $({y^0}, y,\nu, \varphi^0,\varphi) :[0,S]\rightarrow \R\times\R^{n}\times\R\times\R\times\R^m$ satisfying \eqref{extended}.
\end{definition}

\begin{remark}\label{rate}{\rm {We observe that system \eqref{extended} is {\it rate independent}}. {Indeed,} if one considers a  bi-Lipschitz increasing map $\sigma:[0,S]\to [0,\tilde S]$, then $({y^0},y,\nu)\circ \sigma^{-1}$ is a solution of  \eqref{extended} on $ [0,\tilde S]$ corresponding to the control $(\varphi^0,\varphi)\circ \sigma^{-1}$ if and only if $({y^0},y,\nu)$ is a solution of  \eqref{extended} on $ [0, S]$ corresponding to the control $(\varphi^0,\varphi)$.  Thus, imposing the condition on controls
\begin{equation}\label{can}\frac{d\varphi^0}{ds}(s)+\left|\frac{d\varphi}{ds}(s)\right|=1\quad \text{a.e. $s\in[0,S]$ }
\end{equation}  is a convenient, but arbitrary,
{choice}\,\footnote{For instance, one could have used also  controls verifying  $\displaystyle\left(\frac{d\varphi^0}{ds},\frac{d\varphi}{ds}\right)(s)\in \CC$ with   $$\CC\colonequals\left\{(w^0,w)\in \R_+\times \mathcal{C}  :  \  |(w^0, w) | = 1\right\}\quad\mbox{or}\quad \CC:=\left\{(w^0,w)\in \R_+\times \mathcal{C}  :  \  |w^0| +|w^1|+\ldots + |w^m | = 1\right\}$$ or even 
$\CC:=\left\{(w^0,w)\in \R_+\times \mathcal{C}  :  \  \alpha< w^0 +|w|\le\beta\right\}$,  with any $\alpha$, $\beta$ such that  $0<\alpha<\beta$ (with the latter choice,  different pairs may well represent  the same control $u$).}}. 
  \end{remark}

There is a natural  embedding ${\cal I} $  of the family of strict sense processes  into the family of extended sense processes, which is  expressed by 
$$
 {\cal I}(t_1,t_2, x ,v,  u )\colonequals (S,y^0 , y,\nu ,\varphi^0 ,\varphi ), 
$$
where 
$$
\begin{array}{l}
\displaystyle\sigma(t) \colonequals \int_{t_1}^t\bigg(1+ \bigg|\frac{du}{d \tau} ({\tau})\bigg|\bigg)d{\tau}\,, \quad S:=\sigma(t_2), \\\, \\
\ds(y^0, y,\nu,\varphi^0,\varphi)(s)\colonequals  \big(id, x, v,id, u) (\sigma^{-1}( s)\big) \quad \forall s \in [0,S].
\end{array}
$$
 
Notice, in particular,  that
\begin{equation}\label{vareq}\nu(S)=v(t_2).
\end{equation}
Observe also that the map  $\varphi^0(=\sigma^{-1}):[0,S]\to [t_1,t_2]$ is increasing and  $1$-Lipschitz continuous. Furthermore  it  verifies  $\frac{d\varphi^0}{ds}(s)>0$ for almost every $s\in [0,S]$. Actually, {such mappings provide a} {\it characterization} of extended sense controls $(\varphi^0,\varphi)$ {that} are graphs of absolutely continuous,  strict sense controls $u$:

\begin{lemma}\label{characterization}  The embedding $\cal I$ is injective\footnote{Notice that the injectivity is a consequence of the fact that we require  $(\varphi^0,\varphi)$ to verify   $\frac{d\varphi^0}{ds}(s) +  \left|\frac{d\varphi^0}{ds}(s)\right|= 1$ a.e. in $[0,S]$.}.
Moreover, the  image space of the embedding $\cal I$ comprises  the subclass of extended sense processes
 $(S,{y^0} , y,\nu ,\varphi^0 ,\varphi )$ that satisfy $\frac{d\varphi^0}{ds}(s)>0$ almost everywhere. Actually, for every such extended sense process the map  $\varphi^0 $ is invertible  with inverse $(\varphi^0)^{-1}$ absolutely continuous. The map ${\cal I}^{-1}$ is defined (on the image of $\cal I$) by
$$
{\cal I}^{-1}(S,{y^0} , y,\nu ,\varphi^0 ,\varphi )=\left(t_1,t_2, x ,v,  u \right)\colonequals\big (\varphi^0(0),\varphi^0(S), y\circ (\varphi^0)^{-1} , \nu\circ (\varphi^0)^{-1}, \varphi\circ (\varphi^0)^{-1} \big).
$$
\end{lemma}
See e.g. \cite{AR} for a proof of this standard result.

\subsection{ The extended optimal control problem}
Consider the extended optimal control problem
$$
(P_{e}) \left\{
\begin{array}{l}
\displaystyle
\mbox{Minimize } h(y^0(0),y(0),y^0(S), y(S),\nu(S))
\\ [1.5ex]
 \displaystyle\mbox{over } S> 0,\, ({y^0}, y,\nu, \varphi^0,\varphi)  \in W^{1,1}([0,S];\R\times\R^{n}\times\R\times\R\times \R^m)\,\,\,
\mbox{satisfying }
\\ [1.5ex]
\begin{array}{l}
\displaystyle\frac{d{y^0}}{ds} (s)\,=\, \frac{d\varphi^0}{ds}(s) \ \mbox{ a.e. } s \in [0,S]\,, \\ [1.5ex]
\displaystyle\frac{dy}{ds} (s)\,=\, f(y^0(s), y(s))\frac{d\varphi^0}{ds}(s)+ \sum_{j=1}^{m}g_{j}(y^0(s), y(s))\frac{d\varphi^j}{ds}(s) \ \mbox{ a.e. } s \in [0, S]\,,
\\ [1.5ex]
\displaystyle\frac{d\nu}{ds} (s)\,=\,\left|\frac{d\varphi}{ds}(s)\right|  \ \mbox{ a.e. } s \in [0,S]\,, 
\\ [1.5ex]
   \displaystyle
 \bigg(\frac{d\varphi^0}{ds},\frac{d\varphi}{ds}\bigg)(s)\in \CC 
\  \mbox{ a.e. } s \in [0,S]\,, 
\\ [1.5ex]
\displaystyle \nu(0)=0, \quad \nu(S)\le K,\,\,\quad\Big(y^0(0),y(0) , y^0(S),y(S)\Big)\in \T. \end{array}
\end{array}
\right.$$

\begin{definition}
We say that an extended sense process $(S,{y^0} , y ,\nu, \varphi^0 ,\varphi )$ is  {\em feasible for $(P_{e})$} if   
$$
 \nu(0)=0, \quad \nu(S)\le K, 
 \ \ \Big(y^0(0),y(0) , y^0(S),y(S)\Big)\in \T.
$$
\end{definition}

\begin{definition}\label{deflocmin} A feasible extended sense process $(\bar S,\bar y^0 ,\bar  y ,\bar  \nu, \bar \varphi^0 ,\bar \varphi )$ is said to be an {\em extended sense $L^{\infty}$  local minimizer for $(P_{e})$} if  there exists $\delta >0 $ such that :
\begin{equation}
\label{min}
h(\bar y^0(0),\bar y(0),\bar y^0(\bar S), \bar  y(\bar S),\bar  \nu(\bar S)) \leq  h(y^0(0),y(0),y^0(S), y(S),\nu(S))
\end{equation}
 for all extended sense feasible processes
 $(S,{y^0} , y ,\nu, \varphi^0 ,\varphi )$ satisfying
\begin{equation}
\label{close2}
 d_{\infty}\Big((y^0(0),y^0(S), y,\nu) , (\bar y^0(0),\bar y^0(\bar S), \bar y,\bar\nu)\Big) \leq \delta .\,\,\footnotemark 
\end{equation}
\footnotetext{As in the strict sense case,  we mean that  $(  \bar y , \bar \nu)$  and $( y,\nu) $  are  continuously extended to $\R$  so that  they are constant outside the original domains  $[0, \bar S]$ and $[0,S]$, respectively.
{Let us observe, incidentally,  that
$$
S=\varphi^0(S)-\varphi^0(0)+\nu(S)=y^0(S)-y^0(0)+\nu(S).
$$
}}

  If \eqref{min} is satisfied for all  extended sense feasible processes  $(S,{y^0} , y ,\nu, \varphi^0 ,\varphi )$, we say that  $(\bar S,\bar y^0 ,\bar  y ,\bar \nu, \bar \varphi^0 ,\bar \varphi )$ is an {\em extended sense $L^{\infty}$ (global) minimizer}.  
\end{definition}
 
Let us remark that the  notion of extended sense $L^{\infty}$  local minimizer for $(P_{e})$ is consistent with the definition of $L^{\infty}$  local minimizer for problem (P). To be precise, it is not difficult to prove the following result.
\begin{lemma}\label{Min=} A process $(\bar t_1,\bar t_2,\bar x,\bar v, \bar u)$ is a $L^{\infty}$ local minimizer for problem (P) if and only if
$$
(\bar S,\bar y^0,\bar  y,\bar \nu  ,\bar \varphi^0,\bar \varphi):={\mathcal I}(\bar t_1,\bar t_2,\bar x, \bar v, \bar u)
$$
 is an  extended sense $L^{\infty}$  local minimizer for problem $(P_e)$  among  {\em embedded strict  sense} feasible processes.    Moreover,  these equivalent properties imply 
  $$
 h(\bar y^0(0),\bar y(0),\bar y^0(\bar S), \bar  y(\bar S),\bar  \nu(\bar S))=h(\bar t_1,\bar x(\bar t_1), \bar t_2, \bar x(\bar t_2),\bar v(\bar t_2)).
 $$
\end{lemma}

  \begin{remark}{\rm     
Although the question of existence  is not addressed in this paper, we point out that   the existence of an optimal extended sense feasible  process has  been established in several cases. For instance, when {\ either  $K<+\infty$ or $h(t_1, x_1,t_2,  x_2,v)\ge \psi(v)$ with $\displaystyle\lim_{v\to+\infty}\psi(v)=+\infty$} and the target is of the form $\T = \{(\bar t_1,\bar x_1,\bar t_2)\}\times \tilde\T$ for a given closed set $\tilde\T\subset\R^n$,  one can  establish existence  by compactness and the continuity --in   suitable topologies--  of the input-output map\footnote{Of course one has to assume that the set of feasible extended sense trajectories is non-empty.}(see, e.g.,  \cite{BR}, \cite{MR}, \cite{MiRu}, \cite{GS}).  These include so-called  {\it weakly coercive} problems, namely those problems where the cost  has the form $J(t_1,t_2,x,v,u)=\int_{t_1}^{t_2}[\ell_0(t,x(t))+\ell_1(t,x(t))|\dot u(t)|]\,dt$, with $\ell_1(t,x)\ge C$ for some constant $C>0$. }
 \end{remark}
\section{Necessary Conditions for the Extended Optimal Control Problem}\label{SNC}
This section provides necessary conditions, in the form of a Pontryagin Maximum Principle (PMP), for  extended sense  $L^{\infty}$ local minimizers.
\begin{theorem}
\label{PMPe}
Take an extended sense $L^{\infty}$ local minimizer for $(P_{e})$,
$(\bar S,\bar y^0 ,\bar  y ,\bar\nu, \bar \varphi^0 ,\bar \varphi )$. Assume hypothesis {\rm (H1) } is satisfied. Then the following conditions  are  verified: there exist $({p_0}, p)  \in W^{1,1}\left([0,\bar S];\R^{1
+n}\right)$  and real numbers  $\pi$, $\lambda$, with  $\pi\leq 0$ and $\lambda\geq 0$, such that
\begin{equation}\label{fe1}
(p_0 , p , \lambda) \not= (0, 0,0)\,,
\end{equation}
\begin{equation}\label{fe2}
\left\{
\begin{array}{l}
\displaystyle  \frac{dp_0}{ds} (s)\,=\,- p(s)\cdot \bigg( \frac{\partial f}{\partial t}(\bar y^0(s), \bar y(s))\frac{d\bar \varphi^0}{ds}(s) +  \sum_{j=1}^{m} \frac{\partial g_{j}}{\partial t} (\bar y^0(s), \bar y(s))\frac{d\bar\varphi^j}{ds}(s)\bigg)\;\; \mbox{ a.e. } s \in [0,\bar S]\,, \\
\displaystyle  \frac{dp}{ds} (s)\,=\,- p(s)\cdot \bigg( \frac{\partial f}{\partial x}  (\bar y^0(s), \bar y(s))\frac{d\bar \varphi^0}{ds}(s) +  \sum_{j=1}^{m} \frac{\partial g_{j}}{\partial x} (\bar y^0(s), \bar y(s))\frac{d\bar\varphi^j}{ds}(s)\bigg)\;\; \mbox{ a.e. } s \in [0,\bar S]\,,
\end{array}
\right.
\end{equation}
\begin{equation}\label{fe3}
\begin{array}{l}
\displaystyle p(s)\cdot \bigg(f(\bar y^0(s),\bar y(s))\frac{d\bar\varphi^0}{ds}(s) + \sum_{j=1}^mg_{j}(\bar y^0(s), \bar y(s))\frac{d\bar\varphi^j}{ds}(s)\bigg)
+ {p_0}(s)\,\frac{d\bar\varphi^0}{ds}(s) + \pi \left|\frac{d\bar\varphi}{ds}\right|(s) = \\ [1.5ex] 
\qquad\displaystyle \underset{(w^0,w) \in \CC}{\max}
\bigg\{ p(s)\cdot \bigg(f(\bar y^0(s),\bar y(s))w^0  +  \sum_{j=1}^m g_{j}(\bar y^0(s),\bar y(s))w^{j}\bigg) 
+
{p_0}(s)   \,w^0   + \pi|w|
\bigg\}\,=  0
\end{array}
\end{equation}
$\, \mbox{ a.e. }s \in [0,\bar S]$  and
\begin{equation}\label{fe4}
\begin{array}{l}
(p_0(0), p(0),-p_0(\bar S), -p(\bar S),-\pi) \in   \\ [1.5ex] 
\quad\lambda \nabla h(\bar y^0(0), \bar y(0),\bar y^0(\bar S), \bar y(\bar S),  \bar \nu(\bar S))+N_{{\mathcal{T}}\times[0,K]}(\bar y^0(0), \bar y(0),\bar y^0(\bar S), \bar y(\bar S),  \bar \nu(\bar S))\,.
\end{array}
\end{equation}
 Moreover, {additional `multiplier' information is available in the following special cases:}
 \begin{itemize}
 \item[(i)]  if {\, $\lambda\,\frac{\partial h}{\partial v}(\bar y^0(0), \bar y(0),\bar y^0(\bar S), \bar y(\bar S),  \bar \nu(\bar S))=0$\,  and  \,  $\bar\nu(\bar S) <K$} 
 then $$\pi =  0; $$  
  \item[(ii)]  if $\bar y^0(0)< \bar y^0(\bar S)$,  inequality \eqref{fe1} can be strengthened  to 
\begin{equation}\label{strongfe1}
  (p , \lambda) \not= (0,0). 
  \end{equation}  
   \end{itemize}
\end{theorem}

\begin{remark}{\rm The standard PMP applied to the extended problem tells us that there exists a non-trivial multiplier set $({p_0} , p , \pi, \lambda)$ satisfying conditions \eqref{fe2}-\eqref{fe4}. The novelty of  Thm.  \ref{PMPe} consists in  the stronger non-triviality condition \eqref{fe1} and
the  additional information concerning the multipliers non-triviality   in some cases of interest. 
 }
\end{remark}

\begin{proof}[Proof of  Thm.   \ref{PMPe}]    $(P_e )$ is a standard  optimal control problem, to which   the  `free end-time' PMP  is applicable (see, e.g., \cite[Thm. 8.7.1]{Vinter}), with reference to the   
$L^{\infty}$ local extended sense minimizer
$(\bar S,\bar y^0 ,\bar  y ,\bar \varphi^0 ,\bar \varphi )$. This
yields the existence of   $({p_0} , p )  \in W^{1,1}\left([0,\bar S];\R^{1
+n}\right)$ and  $\pi\in\R$, $\lambda\ge0$ satisfying \eqref{fe2}--\eqref{fe3}, the transversality conditions   \eqref{fe4}
 and the  non-triviality  condition
\begin{equation} 
\label{nontriv1}
(p_0 ,p ,  \pi,\lambda)\not= 0\,.
\end{equation}  
Since  {  $\frac{\partial h}{\partial v}(\bar y^0(0), \bar y(0),\bar y^0(\bar S), \bar y(\bar S),  \bar \nu(\bar S))\ge0$ and}  $N_{[0,K]}(v)=\{0\}$  for $v<K$, $N_{[0,K]}(K)= {[0,+\infty)}$, it follows from  \eqref{fe4} that $\pi= 0$ as soon as {  $\lambda\frac{\partial h}{\partial v}(\bar y^0(0), \bar y(0),\bar y^0(\bar S), \bar y(\bar S),  \bar \nu(\bar S))=0$ and    $\bar\nu(\bar S) <K$},  while $\pi\le 0$ in the other cases.   So the proof concerning the sign of $\pi$, and also relation (i),  is  complete.
\vsm

\noindent
 Next, we show  that $(p_0 , p , \lambda)\ne(0,0,0)$.  Indeed, if this were not true, it would follow from (\ref{nontriv1}) that $\pi \not= 0$. But then  $\bar\nu(\bar S) =K$. Integrating the first equation in \eqref{fe3},  one obtains 
$\pi K =0$. This is not possible, since $K>0$.
\vsm

\noindent
Finally, we consider relation (ii). Suppose then that {$\bar y^{0}(0)< \bar y^{0}(S)$}. We must show that  $(p , \lambda) \not= (0,0)$.  If this were not true, we would be able to deduce from {(\ref{fe2}) and (\ref{fe3})} that 
{$p_{0}(\cdot)$ is a constant function and}

$$
 \displaystyle  {p_0}(s)\,\frac{d\bar\varphi^0}{ds}(s) + \pi \left|\frac{d\bar\varphi}{ds}(s)\right| = \underset{(w^0,w) \in \CC}{\max}
\bigg\{ 
{p_0}(s)   \,w^0   + \pi|w|
\bigg\}\,=  0\, \mbox{ a.e. }s \in [0,\bar S].
$$
If $\pi<0$, it would follow from this relation that $p_0(s) = 0$ for all $s\in [0,\bar S]$. This cannot be true since, as we have shown,  
$(p_0 , p , \lambda)\ne(0,0,0)$. If, on the other hand, $\pi=0$, it would follow from the preceding relation that  $p_0(s)<0 $ $ \mbox{ a.e. }s \in [0,\bar S]$. But then we would have $\frac{d\bar\varphi^0}{ds}(s) =0$  
a.e.,
which would imply
$$\bar y^0(S)-\bar y^0(0) =\int_0^{\bar S} \frac{d\bar\varphi^0}{ds}(s)\,ds =0.
$$
This is not possible since, as we have assumed,  $\bar y^0(0)< \bar y^0(\bar S)$. From this contradiction we deduce relation (ii).
\end{proof}

   \section{`No Infimum Gap' and Normality }\label{SNGN}
     
 Write $J(t_1,t_2, x , v,  u )$ for the cost of a strict sense process $(t_1,t_2, x , v, u )$ in problem $(P)$, namely,
$$J(t_1,t_2, x , v,  u ):= h(t_1,x(t_1),t_2, x(t_2) , v(t_2)),$$
 and 
 $J_{e}(S,y^0 , y , \nu, \varphi^0 ,\varphi )$ for the cost of an extended sense process $(S,y^0 , y , \nu, \varphi^0 ,\varphi )$ in problem $(P_{e})$:
 $$J_{e}(S,y^0 , y , \nu, \varphi^0 ,\varphi ):= h(y^0(0), y(0),y^0(S), y(S) ,\nu(S)).$$
 Let us also write ${{\cal A}}$ and ${{\cal A}}_{e}$  for the class of feasible strict sense processes (in problem $(P)$) and  for  the class of feasible extended sense processes (in problem $(P_{e})$), respectively.  
  
  \begin{definition}\label{NiG} We shall say that   
  \begin{itemize}
  \item[{\rm (i)}]
 there is  {\em no  infimum gap} if 
\bel{gap}
\begin{array}{l}
\inf \Big\{J_{e}(S,y^0 , y , \nu, \varphi^0 ,\varphi )\,|\, (S,y^0 , y , \nu, \varphi^0 ,\varphi ) \in {{\cal A}}_{e}\Big\} \\[1.5ex]
\qquad\qquad\qquad\qquad\qquad\qquad \qquad =  \inf \{J(t_1,t_2,x , v, u )\,|\, (t_1,t_2, x , v,   u ) \in {{\cal A}}\}.
\end{array}
\eeq

 \end{itemize}

Furthermore, if $ (\bar S,\bar y^0 ,\bar y , \bar\nu, \bar\varphi^0 ,\bar\varphi )$ is an  extended sense $L^{\infty}$ local  minimizer, we shall say that 
\begin{itemize}
\item[{\rm (ii)}] there is  {\em no local infimum   gap at $(\bar S,\bar y^0 ,\bar y , \bar\nu, \bar\varphi^0 ,\bar\varphi )$}    if, for some $\delta>0$, 
 \begin{eqnarray}\label{uguale}
\nonumber
 J_{e}(\bar S,\bar y^0 ,\bar y , \bar\nu, \bar\varphi^0 ,\bar\varphi )  =\,  \inf \Bigg\{J(t_1,t_2,x , v, u )\,|\,\,\,\, (t_1,t_2, x , v,   u ) \in B_\delta\Big[(\bar S,\bar y^0 ,\bar y , \bar\nu, \bar\varphi^0 ,\bar\varphi ) \Big]\Bigg\}
 \end{eqnarray} 
where we have set 
$$B_\delta\Big[(\bar S,\bar y^0 ,\bar y , \bar\nu, \bar\varphi^0 ,\bar\varphi ) \Big] := 
\Bigg\{\begin{array}{l}(t_1,t_2, x , v,   u ) \in {{\cal A}}\,\,|\,\, (S,y^0 , y , \nu, \varphi^0 ,\varphi )= {\cal I} (t_1,t_2,x , v, u )
 \\
\qquad \hbox{and }   d_{\infty} \Big((\bar y^0(0) ,\bar y^0(\bar S), \bar y, \bar\nu),(y^0(0) , y^0(S), y, \nu)\Big) \,<\,\delta\end{array}\Bigg\}.$$
 \end{itemize}
 \end{definition}      

 { To prove Theorem \ref{cor} below, it will be convenient to introduce the subset ${{\cal A}}^+_{e}\subset {{\cal A}}_{e}$ defined by
 $$
{{\cal A}}^+_{e}:=\left\{ (S,y^0 , y , \nu, \varphi^0 ,\varphi )\,\,|\,\,   (S,y^0 , y , \nu, \varphi^0 ,\varphi )={\cal I} (t_1,t_2,x , v, u )  \text{ and }    (t_1,t_2,x , v, u )\in {{\cal A}}\right\}.
$$
Using the same notation of Def. \ref{NiG},  by   Lemma \ref{Min=} we have:
$$
 \inf \{J(t_1,t_2,x , v, u )\,|\, (t_1,t_2, x , v,   u ) \in {{\cal A}}\}= \inf \{J_e(S,y^0 , y , \nu, \varphi^0 ,\varphi  )\,|\, (S,y^0 , y , \nu, \varphi^0 ,\varphi ) \in {{\cal A}}_e^+\} 
 $$
 and
 $$
 \begin{array}{l}
  \inf \Big\{J(t_1,t_2,x , v, u )\,|\,\,\,\, (t_1,t_2, x , v,   u ) \in B_\delta\big[(\bar S,\bar y^0 ,\bar y , \bar\nu, \bar\varphi^0 ,\bar\varphi ) \big]\Big\} \\[1.5ex] 
\qquad\qquad\qquad= \inf \{J_e(S,y^0 , y , \nu, \varphi^0 ,\varphi  )\,|\, (S,y^0 , y , \nu, \varphi^0 ,\varphi ) \in {{\cal A}}_e^+  \text{ and } \\[1.5ex]  
\qquad\qquad\qquad  \qquad\qquad\qquad d_{\infty} \Big((\bar y^0(0) ,\bar y^0(\bar S), \bar y, \bar\nu),(y^0(0) , y^0(S), y, \nu)\Big) \,<\,\delta\}.
\end{array}
 $$
  }

 \begin{theorem}
 \label{cor} Assume Hypothesis
{\rm (H1)} is satisfied.
\begin{itemize}
\item[{\rm (i)}] Suppose that there exists an $L^{\infty}$  minimizer for $(P_{e})$ which is a normal extremal.  Then there is no infimum gap.
\item[{\rm (ii)}] Take an extended sense $L^{\infty}$  local minimizer $(\bar S,\bar y^0 ,\bar y , \bar\nu, \bar\varphi^0 ,\bar\varphi )$   for $(P_{e})$. Suppose that $(\bar S,\bar y^0 ,\bar y , \bar\nu, \bar\varphi^0 ,\bar\varphi )$ is a normal extremal. Then there is no local infimum gap at \linebreak $(\bar S,\bar y^0 ,\bar y , \bar\nu, \bar\varphi^0 ,\bar\varphi )$.
\end{itemize}

  \end{theorem}

 Theorem \ref{cor}  will be proved in Subsection \ref{secpr} as a consequence  of Theorem \ref{mainextended} below, which relates `isolated processes' and normality.
 
\subsection{Isolated Extended Sense Feasible Processes}

  Notice  that,  on the one hand, the trajectory of any extended  sense process $(\bar S,\bar y^0 ,\bar  y ,\bar\nu, \bar \varphi^0 ,\bar \varphi )$ can be approximated by the trajectory of  an embedded strict sense process, that is,  for any $\delta >0$ we can find an embedded strict sense process $(S, y^0 ,  y , \nu, \varphi^0 , \varphi )$ 
such that {
\begin{equation}
\label{eclose1}
d_{\infty} \Big((\bar y^0(0) ,\bar y^0(\bar S), \bar y, \bar\nu),(y^0(0) , y^0(S), y, \nu)\Big) \,<\,\delta\,,
\end{equation}
where $d_{\infty}$ is the metric defined in \eqref{dinfty}.
On the other hand, the presence of endpoint constraints could make such an approximation  unachievable, if we keep the requirement that  approximating embedded strict sense processes  have to be {\it feasible} as well \,\footnote{ Actually, such  a requirement seems minimal if one wishes feasible extended processes to retain the physical meaning of limits of actual --i.e. strict sense feasible-- processes.}. 
This is because the perturbation that changes  $(\bar S,\bar y^0 ,\bar  y ,\bar\nu, \bar \varphi^0 ,\bar \varphi )$ into  a close embedded strict sense process  might violate  either the end-point constraints or the total variation bound. The phenomenon is captured in the following definition (of topological nature):

\begin{definition}\label{IP} {\em (Isolated feasible extended sense processes)}
We say that a feasible extended sense process $(\bar S,\bar y^0 ,\bar  y , \bar\nu, \bar \varphi^0 ,\bar \varphi )$ is {\em isolated} if, for some $\delta>0$, there does not exist a feasible embedded strict sense process 
$(S,{y^0} , y , \nu, \varphi^0 ,\varphi )$ such that  \eqref{eclose1} is satisfied. 
\end{definition}
A feasible extended sense process having the properties described in this definition is called `isolated' because the state trajectory component $(\bar y^{0},\bar y,\bar\nu) :[0,\bar S] \rightarrow \R\times\R^n\times\R$ is not in the closure (w.r.t. the metric defined by the left side of  \eqref{eclose1}) of the set of state trajectories arising from embedded strict sense processes, whose endpoints lie in the target set and satisfy the total variation constraint.
\vsm

In this section we give a necessary condition for  a feasible extended sense process to be isolated. The relevance  of this condition, which will be explored in subsequent sections,  is the insights 
 that it will provide into possible differences between the infimum cost of the optimal control problem $(P)$ and its extension $(P_{e})$. 
 \begin{definition} {\em (Normal and abnormal extremals)} A feasible extended sense
 process  \linebreak $(\bar S,\bar y^0 ,\bar  y , \bar\nu, \bar \varphi^0 ,\bar \varphi )$,  is said to be an {\em extended sense extremal}  if there exists a set of multipliers $(p_0 ,p ,  \pi,\lambda)$ such  that $(p_0 ,p , \pi,\lambda)$ and
   $(\bar S,\bar y^0 ,\bar  y , \bar\nu, \bar \varphi^0 ,\bar \varphi )$    satisfy the conditions listed in  Thm.  \ref{PMPe}.  An extended sense extremal $(\bar S,\bar y^0 ,\bar  y , \bar\nu, \bar \varphi^0 ,\bar \varphi )$ is said to be {\rm  normal}
   if the only possible choices of multiplier sets $(p_0 ,p , \pi,\lambda)$ associated with it are such that $\lambda >0$. If there exists at least one set of multipliers such that $\lambda =0$, we say that the  extended sense extremal is {\rm abnormal}.
\end{definition}
 
\noindent
The following result relates isolated  extended sense processes and abnormality, from which  the main result (Theorem \ref{cor}) will follow.
 \begin{theorem}\label{mainextended}  If a  feasible extended sense process  $(\bar S,\bar y^0 ,\bar  y , \bar\nu, \bar \varphi^0 ,\bar \varphi )$ is isolated,  then  it is an abnormal extremal for $(P_{e})$. \end{theorem}

\begin{proof}  
\noindent  
We may assume, without loss of generality, 
 that   (H1)\,(i)   is replaced by the  stronger hypothesis:
 
\begin{itemize}
\item[{\bf (H1)*}: (i)]  the vector fields  $f :\R\times\R^{n} \rightarrow \R^{n}$, $g_{j} :\R\times \R^{n} \rightarrow \R^{n},\,j=1,\ldots,m$  are of class   $C^{1}$, Lipschitz continuous and bounded.  
\end{itemize}  
 
\noindent (If only (H1)\,(i)  is satisfied, we consider an optimal control problem related to (P), in which   $f $ and the $g_{j} $'s are modified outside a ball in $\R\times\R^{n}$ containing Graph$\{(\bar y^0 , \bar y  )\}$ in its interior, by truncation and mollification. 
Since the analysis involves consideration    of extended sense trajectories with graphs arbitrarily close to Graph$\{(\bar y^0 , \bar y )\}$ in the Hausdorff sense and the relations appearing in the statement of the theorem concern properties of the data  `near' Graph$\{(\bar y^0 , \bar y  )\}$, it suffices to prove the assertions for only the modified problem.)
\vsm

\noindent
Let $(\bar S,\bar y^0 ,\bar  y , \bar\nu, \bar \varphi^0 ,\bar \varphi )$ be an arbitrary isolated  extended sense feasible process. Define  the map
$\phi:\R\times\R^{n}\times\R\times \R^{n}\times\R \to \R$, given by
\begin{equation}
\label{phi}\begin{array}{l}
\phi (  y_0^0,y_0,  y^0_1, y_1, v)
\colonequals 
\max\big\{
d_{\mathcal{T}}(y_0^0,y_0, y^0_1, y_1),\, (v-K)\vee 0\big\}\,.\end{array}
\end{equation}  
Take a positive sequence $(\epsilon_{i})$  such that $\epsilon_{i}\searrow 0$. {For} each  $i\in\N$, let us consider the free end-point optimal control problem
$$
(\hat P_{i})
 \left\{
\begin{array}{l}
 \qquad\qquad\qquad\quad\quad\mbox{ Minimize }\  \phi \bigg(y^0(0), y(0), y^0(\bar S), y(\bar S),  \nu(\bar S)\bigg)
\\ [1.5ex]
\displaystyle\mbox{over } ({y^0}, y, \nu)  \in W^{1,1}([0,\bar S];\R\times \R^{n}\times\R) \mbox{ and  measurable   functions}  
\,d \,, w 
\displaystyle\mbox{ satisfying }
\\ [1.5ex]
\displaystyle  \frac{d{y^0}}{ds} (s)\,=\,  {(1+ d(s))}(1-|w(s)|) \mbox{ a.e. } s \in [0,\bar S]\,,
\\ [1.5ex]
\displaystyle \frac{dy}{ds} (s)\,=\, (1+ d(s)) \bigg( f(y^0(s), y(s))(1-|w(s)|) + \sum_{j=1}^{m}g_{j}(y^0(s),y(s))w^j(s)\bigg) \mbox{ a.e. } s \in [0,\bar S]\,,
\\ [1.5ex]
\displaystyle  \frac{d{\nu}}{ds} (s)\,=\,  (1+ d(s))|w(s)| \,,\mbox{ a.e. } s \in [0,\bar S]\,,
\\ [1.5ex]
\displaystyle  w(s)\in (1-\epsilon_{i})(\mathcal{C}\cap \B_m)
 \,,\mbox{ a.e. } s \in [0,\bar S]\,,
\\ [1.5ex]
 \displaystyle  d(s) \in [-0.5, +0.5] \,, \mbox{ a.e. } s \in [0,\bar S]\,,
\\ [1.5ex]
  \quad k(0)=0
\end{array}
\right.$$
We call a collection of functions $({y^0} , y ,   \nu , d , w )$  satisfying the constraints in this problem
 a {\it  feasible process for problem $(\hat P_{i})$}.  (Notice that, to simplify the notation, we use here $w$ for the $s$-derivative of $\varphi$, and express the variable $(w^0,w)$  satisfying the constraint $(w^0,w)\in \CC$, as $(1-|w|,|w|)$ where  $w \in (1-\epsilon_{i})(\mathcal{C}\cap \B_{m})$).
 \vsm

\noindent
  For every $i\in\N$,  let $(\bar S,\hat y^0_{i} , \hat y_{i} ,  \hat \nu_{i} , \hat\varphi^0_{i}  ,\hat\varphi_{i} )$ be the extended sense process  in which  $$(\hat y^0_{i},\hat y_{i},  \hat \nu_{i} , \hat\varphi^0_{i},\hat\varphi_{i} )(0)  = (\bar y^0, \bar y, \bar\nu,  \bar\varphi^0,\bar\varphi)(0)$$  and, 
for a.e. $s\in[0, \bar S]$,  
$$ 
\left( \frac{d \hat\varphi^0_i}{ds} , \frac{d \hat\varphi_i}{ds} \right)  (s) \,:=\left\{
 \begin{array}{l} \left(\epsilon_i, (1- \epsilon_i)\, \frac{d \bar\varphi}{ds}(s)\big/\big|\frac{d \bar\varphi}{ds}(s)\big| \right) \ \text{if $ \frac{d \bar\varphi^0}{ds}(s) <\epsilon_i$,} \\ \, \\
 \left( \frac{d \bar\varphi^0}{ds}(s) , \frac{d \bar\varphi}{ds}(s)\right)  \ \text{if $ \frac{d \bar\varphi^0}{ds}(s) \ge\epsilon_i$. }
 \end{array}\right. 
 $$
Notice that for each $i\in\N$, on the one hand,  $\ds \frac{d \hat\varphi_i}{ds}(s)\in(1-\epsilon_i)(\C\cap\B_m)$ a.e. $s\in [0,\bar S]$, so  that $(\hat y_{i}^0 , \hat y_{i} ,\hat \nu_i , 0,\frac{d\hat\varphi_i}{ds})$  is a feasible process for  problem $(\hat P_{i})$; on the other hand,  $\ds \frac{d \hat\varphi_i^0}{ds}(s)\ge \epsilon_i>0$ a.e. $s\in[0 \bar S]$, so that $(\bar S,\hat y^0_{i} , \hat y_{i} ,  \hat \nu_{i} , \hat\varphi^0_{i}  ,\hat\varphi_{i} )$ is an embedded strict sense  process for problem $(P_{e})$. Moreover, 
 \begin{equation}
\label{A2}
\left\| \left(\frac{d \hat\varphi^0_i}{ds},\frac{d \hat\varphi_i}{ds}\right) - \left(\frac{d \bar\varphi^0}{ds}, \frac{d \bar\varphi}{ds}\right) \right\|_{  L^{\infty}(0,\bar S)} \rightarrow 0  \mbox{ as } i \rightarrow \infty\,,
\end{equation} 
which implies that the $( \hat\varphi^0_{i},\hat\varphi_{i} )$ converge uniformly to $(\bar\varphi^0, \bar\varphi)$. {Noting
Hypothesis (H1)* \,(i) 
 and also the continuity properties of the input-output map, proved  e.g.in  \cite{MR},  we obtain:}
\begin{equation}
\label{A1}
\|(\hat y^0_{i},\hat y_{i},  \hat \nu_{i}) -(\bar y^0,\bar y,\bar\nu) \|_{L^{\infty}(0,\bar S)} \rightarrow 0, \mbox{ as } i \rightarrow \infty. 
\end{equation}
 Since $\phi$ is non-negative valued and vanishes at
$(\bar y^0(0), \bar y(0), \bar y^0(\bar S), \bar y(\bar S), \bar \nu(\bar S))$, and in view of  \eqref{A1}, there exists a sequence  $\rho_{i} \searrow 0$ such that, for each $i$,
  $(\hat y_{i}^0 , \hat y_{i} ,\hat \nu_i , 0,\frac{d \hat\varphi_i}{ds}  )$  is a $\rho_{i}^{2}$-minimizer for $(\hat P_{i})$, i.e. is a process whose cost  exceeds the infimum cost for  $(\hat P_{i})$ by an amount not greater that $\rho_{i}^{2}$.
\vv

\noindent
Problem $(\hat P_{i})$ can be regarded as an optimization problem with continuous cost over elements $(d , w ,   y^0(0), y(0))$ in a closed subset of $L^{1}([0, \bar S]; \R)\times L^{1}([0, \bar S];\R^m)\times \R\times \R^{n}$. In consequence of Ekeland's Principle \cite[Thm. 3.3.1]{Vinter}, there exists, for each $i$, a feasible process  for $(\hat P_{i})$, \newline $(y^0_{i} , y_{i} ,\nu _{i} ,  d_{i} ,w_{i} )$,  such that
\begin{equation}
\label{A3} |(  \hat y^0_{i} ,\hat y_{i})(0)- ( y^0_{i} ,y_{i})(0) |  \rightarrow 0, \,
\left\| \hat w_i -w_i \right\|_{L^{1}(0, \bar S)} \rightarrow 0 \mbox{ and } \|d_{i} \|_{L^{1}(0, \bar S)} \rightarrow 0  \mbox{ as } i \rightarrow \infty \,,
\end{equation}
and 
  $ (y^0_{i} , y_{i} , \nu_{i} , d_{i} , w_{i} )$ is a minimizer for 
$$
(P_{i}) \left\{\begin{array}{l}
\mbox{ Minimize } \phi \big(y^0(0), y(0), y^0(\bar S),  y(\bar S),   \nu(\bar S)\big) + \\ [1.5ex]
\ds \qquad\qquad\qquad\rho_{i}\cdot\bigg( |( y^0,y)(0)-( y^0_i,y_{i})(0)| +\ds\int_0^{\bar S} |w(s)-w_{i}(s)|+ |d(s)-d_{i}(s)|ds\bigg)
\\ [1.5ex]
\displaystyle\mbox{over } (y^0, y,  \nu)  \in W^{1,1}([0,\bar S];\R\times \R^{n}\times \R) \mbox{ and  measurable   functions}  
\\ [1.5ex] d :[0,\bar S]\rightarrow \R\,,  \ w :[0,\bar S]\rightarrow \R^m\,
\displaystyle\mbox{ satisfying }
\\ [1.5ex]
\displaystyle  \frac{d{y^0}}{ds} (s)\,=\,  {(1+ d(s))}(1-|w(s)|) \mbox{ a.e. } s \in [0,\bar S]\,, 
\\  [1.5ex]
\displaystyle \frac{dy}{ds} (s)\,=\, (1+ d(s)) \bigg( f(y^0(s), y(s))(1-|w(s)) + \sum_{j=1}^{m}g_{j}(y^0(s), y(s))w^j(s)\bigg) \mbox{ a.e. } s \in [0,\bar S]\,,
\\ [1.5ex]
 \displaystyle  \frac{d{\nu}}{ds} (s)\,=\, (1+ d(s)) |w(s)|
\\ [1.5ex]
\displaystyle  w(s)\in (1-\epsilon_{i})(\mathcal{C}\cap \B_m)
 \,,\mbox{ a.e. } s \in [0,\bar S]\,,
\\ [1.5ex]
 \displaystyle  d(s) \in [-0.5, +0.5]  \mbox{ a.e. } s \in [0,\bar S]\,, 
 \\ [1.5ex]
 \nu(0)=0.
\end{array}\right.
$$
It follows from (\ref{A2}) and (\ref{A3}) that
\begin{equation}
\label{A4} 
\left\|w_i - \frac{d\bar\varphi}{ds} \right\|_{L^{1}(0,\bar S)} \rightarrow 0 \mbox{ and }  \| ({y^0_i} , y_{i}, \nu_i )- (\bar y^0 , \bar y, \bar\nu ) \|_{L^{\infty}(0, \bar S)} \rightarrow 0, \, \mbox{ as } i \rightarrow \infty\,.
\end{equation}
By extracting subsequences, we can arrange that
\begin{equation}
\label{A4*} 
w_i(s)\rightarrow \frac{d\bar\varphi}{ds}(s) \mbox{ and } d_{i}(s) \rightarrow 0\quad \mbox{a.e. } s\in [0,\bar S], \mbox{ as } i \rightarrow \infty\,.
\end{equation}
 Now, 
apply the PMP to $(P_{i})$ with reference to the minimizer 
$(  y^0_{i} , y_{i} , \nu_{i} , d_{i} ,  w_{i} )$. This yields $({p_0}_{i},p_{i}) \in W^{1,1}([0,\bar S];\R^{n})$ and  $\pi_i \in\R$ such that the adjoint equations
\begin{small}
\begin{eqnarray}
\label{S0}
&&
\displaystyle\frac{d{p_0}_{i}}{ds} (s)\, =  -p_{i}(s)\cdot \bigg(\frac{\partial f}{\partial t}  ( y^0_{i}(s),y_{i}(s))(1-|w_{i}(s)|) + \sum_{j=1}^m\frac{\partial g_j}{\partial t}  (y^0_{i}(s), y_{i}(s))w_{i}^{j}(s)\bigg)   (1+d_{i}(s)),  \ 
\\ 
\nonumber
&&
\displaystyle\frac{dp_i}{ds} (s)\, =  -p_{i}(s)\cdot \bigg(\frac{\partial f}{\partial x}  (y^0_{i}(s), y_{i}(s))(1-|w_{i}(s)|) + \sum_{j=1}^m\frac{\partial g_j}{\partial x}  (y^0_{i}(s), y_{i}(s))w_{i}^{j}(s)\bigg)   (1+d_{i}(s))  
\end{eqnarray}
\end{small}
are verified (on $[0,\bar S]$), inequality
\begin{eqnarray}
\nonumber
&&
\displaystyle \int_{0}^{\bar S}
\Bigg\{\bigg[
p_{i}(s)\cdot \bigg(f(y^0_{i}(s), y_{i}(s))(1-|w(s)|) + \sum_{j=1}^mg_{j}(y^0_{i}(s), y_{i}(s))w^{j}(s) \bigg)
\\
\nonumber
&&
\hspace{0.3 in}
+ {p_0}_{i}(s)  (1-|w(s)|) +  \pi_i |w(s)|
\bigg]
  (1+ d(s)) 
 + \rho_{i}
\left(
 |w(s)-w_{i}(s)| + |d(s)-d_{i}(s)|
 \right) \Bigg\}ds
\\
\nonumber
&& \hspace{0.3 in} \leq \; \int_{0}^{\bar S}\Bigg\{
\bigg[p_{i}(s)\cdot \bigg(f(y^0_{i}(s), y_{i}(s))(1-|w_{i}(s)|)+ \sum_{j=1}^m g_{j}(y^0_{i}(s), y_{i}(s))w_{i}^{j}(s)\bigg) 
\\
\label{S1}
&&
\hspace{1.0 in}
+ {p_0}_{i}(s) (1-|w_{i}(s)|)+  \pi_i |w_i(s)| \bigg]
  (1+ d_i(s)) \Bigg\}ds
 \end{eqnarray}
holds true for all measurable selectors $w $ and $d $ of   $(1-\epsilon_{i})(\mathcal{C}\cap \B_m)$ and $[-0.5,0.5]$ respectively, and, furthermore, one has the transversality relation
\begin{eqnarray}
\nonumber 
 &&({p_0}_{i}(0),  p_{i}(0), -{p_0}_{i}(\bar S), -p_{i}(\bar S), -\pi_i)  
  \\ [1.5ex]
\nonumber 
 && \in\partial \Big( \phi(y^0_{i}(0),  y_{i}(0),y^0_{i}(\bar S), y_{i}(\bar S),  \nu_i(\bar S)) + \rho_i|(y^0,y)(0) - (y^0_i,y_i)(0)| \Big)\\ [1.5ex]
 \label{S2}
&&
\hspace{1.0 in} \,\subseteq\,   \partial\phi\big(y^0_{i}(0),  y_{i}(0),y^0_{i}(\bar S), y_{i}(\bar S),  \nu_i(\bar S)\big)
 + \rho_{i} \B_{1+n} \times \{0_{1+n}\}\times \{0\}\,.
\end{eqnarray}
 Notice that we have set the cost multiplier equal to $1$, as is permitted, since $(P_{i})$  has free right endpoint.  It can be deduced from (\ref{S0}) and (\ref{S2}) that $\{({p_0}_{i},p_{i}) \}$ is an equi-bounded sequence of functions in $W^{1,1}$ with equi-integrable derivatives and  $\{\pi_i\}$ is bounded. It follows that there exist  $(p_0,p)  \in W^{1,1}$ and   $\pi\in\R$  such that, along some sequence (we do not relabel)
\begin{equation}
\label{A4**}
\|({p_0}_{i},p_{i}) - (p_0,p) \|_{L^{\infty}(0, \bar S)  } \rightarrow 0, \quad \pi_i\to \pi \quad \mbox{ as } i \rightarrow \infty\,.
\end{equation}
Notice next from (\ref{A4}) that
$$
 |(y^0_{i}(0), y_{i}(0),y^0_{i}(\bar S),  y_{i}(\bar S),\nu_{i}(\bar S)) - (\bar y^0(0), \bar y(0), \bar y^0(\bar S), \bar y(\bar S),  \bar \nu(\bar S))| \rightarrow 0    \ \mbox{ as } i \rightarrow \infty\,.
$$
Let $(\varphi^0_i,   \varphi_i) \in W^{1,1}([0,\bar S];\R\times \R^{m})$ be defined by 
$$
\begin{array}{l}
\displaystyle\varphi^0_i (s)\,\colonequals \bar\varphi^0(0) + \int_0^s\, (1+ d_{i}(r))\,(1-| w_{i}|(r))dr  \,\\ \, \\
\varphi_i(s)\,\colonequals \ds\bar\varphi(0) + \int_0^s\, (1+ d_{i}(r))\, w_{i}(r)dr \,
\end{array}\quad  \forall s \in [0,\bar S].
$$
By the rate independence of the system
$$
\begin{array}{l}
\displaystyle  \frac{d y^0_i}{ds} (s)\,=  \, (1+ d_{i}(s))\frac{d\varphi^0_i}{ds}(s)  \,, \\
\displaystyle \frac{dy_i}{ds} (s)\,=\, (1+ d_{i}(s)) \bigg( f(y^0_{i}(s), y_{i}(s))\frac{d\varphi^0_i}{ds}(s) + \sum_{j=1}^m g_{j}(y^0_{i}(s), y_{i}(s)) \frac{d\varphi^j_i}{ds}(s)\bigg)\,,\\  
\displaystyle \frac{d\nu_i}{ds} (s)\,=\, (1+ d_{i}(s))\left|\frac{d\varphi_i}{ds}(s)\right| ,
\end{array}
$$
by considering the reparameterization  $\ds\sigma_{i}(s)= \int_0^s (1+ d_{i}(s'))ds' $,
$$
( \tilde y^0_{i} ,\tilde y_{i}, \tilde \nu_i, {\tilde\varphi}^0_{i}, \tilde{\varphi}_{i})(\sigma):=(y^0_{i}, y_{i}, \nu_i,\varphi^0_{i}, \varphi_{i})(\sigma_{i}^{-1}(\sigma)), \quad 0 \leq \sigma \leq \tilde S,
$$
 where $\tilde S_{i} := \sigma_{i}(\bar S)$, we obtain --see Remark \ref{rate}-- that $( \tilde y^0_i , \tilde y_i ,\tilde \nu_i )$ satisfies the differential system
\begin{equation}
\label{list0}
\begin{array}{l}
  \displaystyle \frac{d\tilde y^0_i}{ds} (s)\,=\,   \frac{d{\tilde\varphi}^0_{i}}{ds} (s)  \quad \mbox{ a.e. } s \in [0,\tilde S_{i}]\,,
 \\
 \displaystyle  \frac{d\tilde y_i}{ds} (s)\,=\,  f( \tilde y^0_i(s), \tilde{y}_{i}(s)) \frac{d{\tilde\varphi}^0_{i}}{ds} (s)  + \sum_{j=1}^m g_{j}(\tilde y^0_i(s), \tilde y_{i}(s))\frac{d\tilde{\varphi}^j_{i}}{ds} (s) \quad \mbox{ a.e. } s \in [0,\tilde S_{i}]\,,\\  
\displaystyle \frac{d\tilde \nu_i}{ds} (s)\,=\, \left|\frac{d\tilde\varphi_i}{ds}(s)\right|  \quad  \mbox{ a.e. } s \in [0,\tilde S_{i}]\,.
  \end{array}
\end{equation}
We observe that, for each $i$,  we have
\begin{eqnarray}
\label{list1}
&&
 \displaystyle   \frac{d{\tilde\varphi}^0_{i}}{ds} (s)=1- \left|\frac{d\tilde{\varphi}_{i}}{ds} (s)\right|, \quad   \frac{d\tilde{\varphi}_{i}}{ds} (s) \in (1-\epsilon_{i})\,(\mathcal{C}\cap \B_m) \ \mbox{ a.e. } s \in [0,\tilde S_i],
 \\
 \label{list2}
 &&
 \displaystyle (\tilde y^0_i(0), \tilde y_{i}(0), \tilde y^0_{i}(\tilde S_{i}), \tilde y_{i}(\tilde S_i),  \tilde \nu_i(\tilde S_{i}))\;= \;(y^0_i(0), y_{i}(0), y^0_{i}(\bar S),  y_{i}(\bar S), \nu_i(\bar S))\,.
\end{eqnarray}
Therefore, we deduce from  (\ref{A4}) that, for $i$ sufficiently large, 
\begin{equation}
\label{list3}
d_\infty\Big((\tilde y^0_i(0),\tilde y^0_i(\tilde S_i), \tilde y_i,\tilde\nu_i), (\bar y^0(0),\bar y^0(\bar S),\bar y , \bar \nu)\Big)<\delta,
\end{equation}
where $\delta>0$ is the constant appearing in Definition \ref{IP}, { with reference to the isolated extended sense feasible process $(\bar S, \bar y^0 , \bar y , \bar\nu,  \bar \varphi^0 ,  \bar \varphi )$.}
Relations (\ref{list0}) and (\ref{list1}) tell us that $(\tilde S_{i}, \tilde y^0_{i}  ,\tilde y_{i} , \tilde \nu_{i} , \tilde\varphi^0_{i} ,\tilde\varphi_{i} )$ is an embedded strict sense process.  But taking note of the conditions imposed on the process $(\bar S, \bar y^0 , \bar y ,  \bar\nu, \bar \varphi^0 ,  \bar \varphi )$ in the theorem statement (namely,  the fact that it is an  isolated feasible extended sense process), we deduce from (\ref{list3})   that
$(\tilde S_{i}, \tilde y^0_{i}  ,\tilde y_{i} , \tilde \nu_{i} , \tilde\varphi^0_{i} ,\tilde\varphi_{i} )$ cannot be a feasible embedded strict sense process. We conclude that it must violate the endpoint constraints. By (\ref{list2}), also $(\bar S, y^0_{i}  , y_{i} , \nu_i,  \varphi^0_{i} ,\varphi_{i} )$ must  violate these constraints. Therefore,
 \begin{equation}
\label{greater} 
\max \{ d_{\mathcal{T}}(y^0_i(0), y_{i}(0), y^0_{i}(\bar S),  y_{i}(\bar S)),\,  (\nu_i(\bar S)-K)\vee 0\} \,>\,0\,.
\end{equation}
{Inequality  (\ref{greater}) provides the important information that, if  the maximum in
$$
\begin{array}{l}
\ds\phi(y^0_{i}(0), y_{i}(0), {y^0_i}(\bar S),  y_{i}(\bar S),  \nu_i(\bar S)) = 
\max\{  d_{\mathcal{T}}(y^0_{i}(0), y_{i}(0), {y^0_i}(\bar S),  y_{i}(\bar S)),\,    (\nu_i(\bar S)-K)\vee 0\}
\end{array}
$$
is achieved at   $d_{\mathcal{T}}(y^0_{i}(0), y_{i}(0), {y^0_i}(\bar S),  y_{i}(\bar S))$  then $d_{\mathcal{T}}(y^0_{i}(0), y_{i}(0), {y^0_i}(\bar S),  y_{i}(\bar S)) >0$, and if the maximum is achieved at $ (\nu_i(\bar S)-K)\vee 0$, then $\nu_i(\bar S)-K >0$.}
Note also that (see, e.g., \cite[Lemma 4.8.3]{Vinter}) 
$$
\begin{array}{l} 
\partial d_{\mathcal{T}}(y^0_{i}(0), y_{i}(0), {y^0_i}(\bar S),  y_{i}(\bar S)) \subseteq \partial \B_{1+n+1+n} \quad\text{ if $(y^0_{i}(0), y_{i}(0), {y^0_i}(\bar S),  y_{i}(\bar S)) \notin \mathcal{T}$,} \\ \, \\
 \partial (\nu_i(\bar S)-K)\vee 0)= 1 \quad\text{ if $\nu_i(\bar S)>K$.}
 \end{array}
 $$
 Making use of these facts, we can deduce from the  `max' rule of subdifferential calculus (see, e.g., \cite[Thm. 5.5.2]{Vinter}) the following estimate: if   $\xi \in \partial\phi(z)$ 
at $z = (y^0_{i}(0), y_{i}(0), {y^0_i}(\bar S),  y_{i}(\bar S), \nu_i(\bar S))$ then 
$$
\xi \;\in\; \alpha_{i}\,
\bigg(\big(\partial d_{\mathcal{T}}(y^0_{i}(0), y_{i}(0), {y^0_i}(\bar S),  y_{i}(\bar S))\cap \partial \B_{1+n+1+n}\big) \times\{0\}\bigg) +(1-\alpha_i)\bigg(\{0_{1+n+1+n}\}\times \{1\}\bigg)
$$
with $\alpha_{i}\in[0,1]$.
 It follows then from (\ref{S2}) that
 \begin{equation}
 \label{S2*}
 \begin{array}{l}
 ({p_0}_i(0), p_{i}(0),-{p_0}_i(\bar S), -p_{i}(\bar S)) \in \\[1.5ex]
 \qquad\qquad\qquad\qquad  \alpha_{i}\,( \partial d_{\mathcal{T}}( y^0_{i}(0), y_{i}(0), {y^0_i}(\bar S),  y_{i}(\bar S)) \cap \partial \B_{1+n+1+n} ) \,
+ \, \rho_{i} \B_{1+n}\times \{0_{1+n}\}
 \end{array}
 \end{equation}
 and
 \begin{equation}
 \label{S2**}
  \pi_i=-(1-\alpha_i).
 \end{equation}
By extracting further subsequences, we can arrange that $\alpha_{i} \rightarrow \alpha$ as $i \rightarrow \infty$, for some $\alpha\in[0,1]$. Passing to the limit in (\ref{S0}) (in `integral' form), in (\ref{S1}) and in (\ref{S2}), as $i \rightarrow \infty$, with the help of (\ref{A4}), (\ref{A4*}) and  (\ref{A4**}), we arrive at
 \begin{eqnarray}
\label{S0*}
&&
\displaystyle -    \frac{dp_0}{ds}(s)\,=\, p(s)\cdot \bigg(\frac{\partial f}{\partial t}  (\bar y^0(s), \bar  y(s)) \frac{d\bar \varphi^0}{ds} (s)+  \sum_{j=1}^{m}\frac{\partial g_j}{\partial t}  (\bar y^0(s), \bar  y(s)) \frac{d\bar \varphi^j}{ds} (s)\bigg),  
\\
\nonumber
&&
\displaystyle -    \frac{dp}{ds}(s)\,=\, p(s)\cdot \bigg(\frac{\partial f}{\partial x}  (\bar y^0(s), \bar  y(s)) \frac{d\bar \varphi^0}{ds} (s)+  \sum_{j=1}^{m}\frac{\partial g_j}{\partial x}  (\bar y^0(s), \bar  y(s)) \frac{d\bar \varphi^j}{ds} (s)\bigg) \ \mbox{ a.e. $s\in [0,\bar S]$,} 
\\
\nonumber
&&\int_{0}^{\bar S}
p(s)\cdot \bigg(f(\bar y^0(s),  \bar y(s))(1-|w(s)|)+ \sum_{j=1}^m g_{j}(\bar y^0(s), \bar y(s)) w^{j}(s)\bigg)\,(1+ d(s))  ds 
\\
\nonumber
&&
 \qquad\qquad\qquad\qquad\qquad\qquad\qquad\qquad+ \int_{0}^{\bar S}\bigg({p_0}(s) \,(1-|w(s)|) +\pi|w| \bigg)
\,(1+ d(s))  ds
\\
\label{MP*}
&&\hspace{0.2 in} \leq \; \int_{0}^{\bar S} 
 p(s)\cdot \bigg(f( \bar y^0(s), \bar y(s))  \frac{d\bar{\varphi}^0}{ds} (s)  + \sum_{j=1}^m g_{j}(\bar y^0(s), \bar y(s))  \frac{d\bar{\varphi}^j}{ds} (s) \bigg)\,ds
 \\
\nonumber
&&
 \qquad\qquad\qquad\qquad\qquad\qquad\qquad\qquad+\int_{0}^{\bar S}
\bigg( {p_0}(s)  \, \frac{d\bar{\varphi}^0}{ds} (s) +\pi \left| \frac{d\bar{\varphi}^0}{ds} (s) \right| \bigg) \,ds
\nonumber
    \,,
 \\
\label{S1*}
&&
\hspace{0.1 in}\mbox{for all selectors $w  $ and $d $ of $\mathcal{C}\cap \B_m$ and $[-0.5,0.5]$ respectively,}
 \\
 \nonumber
 &&
 \\
 \label{S2*}
 && (p^0(0), p(0),-p^0(\bar S),-p(\bar S) \in \alpha  (\partial d_{\mathcal{T}}(\bar y^0(0), \bar y(0),\bar y^0(\bar S), \bar y(\bar S))\cap \partial \B_{1+n+1+n})\,,
 \\
 \nonumber
  &&
 \\
 \label{S3*}
 &&
 |\pi|=1-\alpha.
\vspace{0.2 in}
\end{eqnarray}
 We deduce from (\ref{MP*}) with the help of a measurable selection theorem that, for { a.e. }$s\in [0,\bar S]$,
$$\begin{array}{l} \ds p(s)\cdot \bigg(f(\bar y^0(s), \bar y(s))  \frac{d\bar{\varphi}^0}{ds} (s)  + \sum_{j=1}^m g_{j}(\bar y^0(s), \bar y(s))  \frac{d\bar{\varphi}^j}{ds} (s) \bigg)
+ {p_0}(s)  \, \frac{d\bar{\varphi}^0}{ds} (s) +\pi \left| \frac{d\bar{\varphi}^0}{ds} (s) \right| =\,\\
\underset{(w^0,w)  \in \CC,\;d \in [0.5,1.5]}{\max}
\Big\{ 
\big[p(s)\cdot \big(f(\bar y^0(s), \bar y(s))w^0\; + 

\sum_{j=1}^{m}g_{j}(\bar y^0(s), \bar y(s))w^{j}\big)
+ {p_0}(s) w^0 +\pi|w|
\big] (1+d)\Big\}\, .\end{array}
$$
 
 \noindent
Since 
 $1$ is interior to $[0.5, 1.5]$, this implies
\begin{eqnarray}
\nonumber
&&
 p(s)\cdot \bigg(f(\bar y^0(s),  \bar y(s))  \frac{d\bar{\varphi}^0}{ds} (s)  + \sum_{j=1}^m g_{j}(\bar y^0(s), \bar y(s))  \frac{d\bar{\varphi}^j}{ds} (s) \bigg)
+ {p_0}(s)  \, \frac{d\bar{\varphi}^0}{ds} (s) +\pi \left| \frac{d\bar{\varphi}^0}{ds} (s) \right|=
\\
\label{MP**}
&&
 \underset{(w^0,w)  \in \CC}{\max}
\left\{ 
p(s)\cdot \big(f(\bar y^0(s), \bar y(s))w^0  + 
 \sum_{j=1}^{m} g_{j}(\bar y^0(s), \bar y(s))w^j\big)
+ {p_0}(s)  w^0 +\pi|w|
\right\}\,=0 \mbox{ a.e.}
\end{eqnarray} 
 Furthermore, (\ref{S2*}) and (\ref{S3*}) imply that 
 \begin{equation}
 \label{tc*}
 (p_0(0), p(0),-p_0(\bar S),  -p(\bar S)) \,\in \, N_{\mathcal{T}}(\bar y^0(0), \bar y(0), \bar y^0(\bar S),  \bar y(\bar S)).
 \end{equation}
 From (\ref{S2*}) and (\ref{S3*}) we deduce that $({p_0} , p ,\pi)\not= (0,0,0)$.
 Employing the same arguments as those in the proof of  Thm.  \ref{PMPe} we deduce from the latter relation that
\begin{equation}
 \label{nontriv*}
 (p_0 ,p ) \not= 0\,.
 \end{equation}
 Surveying the relations satisfied by the absolutely continuous function $(p_0,p) $ and   $\pi\le 0$, namely (\ref{S0*}), (\ref{MP**}), (\ref{tc*})  and (\ref{nontriv*}), we see that the proof is complete.
 
\end{proof}

\subsection{Proof of Theorem \ref{cor}.}\label{secpr}
 It is convenient to prove the two assertons in reverse order. 
\vspace{0.1 in}

\noindent
 (ii): Assume, contrary to the assertions of the theorem, that the  $L^{\infty}$ local minimizer   \linebreak $(\bar S,\bar y^0 , \bar y , \bar\nu, \bar \varphi^0 ,\bar \varphi )$ for $(P_{e})$ is a normal extended sense  extremal while, at the same time, there is a local infimum gap at $(\bar S,\bar y^0 , \bar y , \bar\nu, \bar \varphi^0 ,\bar \varphi )$.} We deduce from the latter property that there exists   a number $c>0$ such that
\begin{eqnarray}
\label{gap1}
&& 
J_{e}(\bar S,\bar y^0 ,\bar  y , \bar\nu, \bar \varphi^0 ,\bar \varphi )
\,<\,
 J_{e}(S,y^0 , y , \nu, \varphi^0 ,\varphi ) \,-c\,
 \end{eqnarray}
 for every embedded strict sense  feasible process $(S,y^0 , y , \nu, \varphi^0 ,\varphi )$ verifying \eqref{close2} for the same $\delta>0$ as in Def. \ref{deflocmin}. Since    $(\bar S,\bar y^0 , \bar y , \bar\nu, \bar \varphi^0 ,\bar \varphi )$ is a normal extended sense extremal, we know from  Thm.  \ref{mainextended} that it is not an isolated  extended sense feasible  process. This means that there exists a sequence of  embedded strict sense feasible processes $\{(S_{i},y_{i}^0 , y_{i} , \nu_i,  \varphi_{i}^0 ,\varphi_{i} )\}$ such that  
 \begin{equation}\label{gap2}
|y_{i}^0(0)-\bar y^0(0)|+ |y_{i}^0(S_{i})-\bar y^0(\bar S)|+\|(y_{i},\nu_i ) - ( \bar y, \bar\nu )\|_{L^{\infty}(\R)} \rightarrow 0\,, \mbox{ as } i \rightarrow \infty\,,
\end{equation}
(where $( \bar y , \bar \nu)$  and $( y_i,\nu_i)$
 are extended continuously to $\R$ by  requiring  them to be constant outside the original domains
  $[0, \bar S]$ and $[0,S_i]$, respectively).   This is simply shown to imply  that $\ds\lim_{i \rightarrow \infty} S_i=\bar S$. 
But
\begin{equation}\label{gap20}
J_{e}(S,y_{i}^0 , y_{i} , \nu_i, \varphi_{i}^0 ,\varphi_{i} ) =h(y^{0}_{i}(0), y_{i}(0),y^{0}_{i}(S_{i}), y_{i}(S_{i}), \nu_{i}(S_{i}))\, 
\end{equation}
and, since $h $ is continuous, we deduce {from} (\ref{gap2}) and (\ref{gap20}) that
\begin{eqnarray*}
&&
\lim_{i \rightarrow \infty} J_{e}(S,y_{i}^0 , y_{i} , \nu_i, \varphi_{i}^0 ,\varphi_{i} ) =\lim_{i \rightarrow \infty}  h(y^{0}_{i}(0), y_{i}(0),y^{0}_{i}(S_{i}), y_{i}(S_{i}),  \nu_{i}(S_{i}))
\\
&&
\quad = h(\bar y^{0}(0), \bar y(0),\bar y^{0}(\bar S), \bar y(\bar S), \bar \nu(\bar S))
= J_{e}(\bar S,\bar y^0 ,\bar  y ,\bar\nu, \bar \varphi^0 ,\bar \varphi )\,.
\end{eqnarray*}
This is not possible, in view of (\ref{gap1}).  So there is no local infimum gap at $(\bar S,\bar y^0 , \bar y , \bar\nu, \bar \varphi^0 ,\bar \varphi )$. 
\vspace{0.1 in}
 \noindent
(i): Suppose that there exists a minimizer for $(P_{e})$ (write it $(\bar S,\bar y^0 , \bar y , \bar\nu, \bar \varphi^0 ,\bar \varphi )$ ) which is a normal extended sense extremal. Then $(\bar S,\bar y^0 , \bar y , \bar\nu, \bar \varphi^0 ,\bar \varphi )$ is certainly an extended sense $L^{\infty}$  local minimizer. So by part (ii), there is no local infimum gap at $(\bar S,\bar y^0 , \bar y , \bar\nu, \bar \varphi^0 ,\bar \varphi )$. This means that, for some $\delta>0 $,
\begin{eqnarray*}
&&J_{e}(\bar z) = \inf \{   J_{e}(z) \,|\,  z \in {{\cal A}}_e^+  \mbox{ s.t. } d_{\infty}(z,\bar z) \leq \delta \}
 \geq \inf \{   J_{e}(z) \,|\,  z \in       {{\cal A}}_e^+  \}\,,
\end{eqnarray*}
in which $\bar z := (\bar S,\bar y^0 ,\bar  y ,\bar\nu, \bar \varphi^0 ,\bar \varphi )$. Since $\bar z$ is a minimizer,
$$
J_{e}(\bar z)= \inf \{   J_{e}(z) \,|\,  z \in {{\cal A}}_{e}  \}
$$
Bearing in that $\inf \{   J_{e}(z) \,|\,  z \in {{\cal A}}_{e} \}\leq \inf \{   J_{e}(z) \,|\,  z \in     {{\cal A}}_e^+\}$, we conclude that
$$
\inf \{   J_{e}(z) \,|\,  z \in {{\cal A}}_{e}\}= \inf \{   J_{e}(z) \,|\,  z \in     {{\cal A}}_e^+ \}\,,
$$
i.e. there is no infimum gap.

\qed

\section{Verifiable conditions for No Infimum Gap}\label{SNG}
The sufficient condition of  Thm.  \ref{cor} for the absence of an infimum gap  has the disadvantage, as a practical test, that it is expressed in terms of some minimizer,   detailed information  about which  might not be available. The `normality' test is  of interest, nonetheless, because in certain special cases it can be replaced by  simpler,  verifiable conditions, some examples of which we now provide. These involve two notions of controllability w.r.t. the target set.
\vspace{0.1 in}
\begin{definition} Consider the
control system 
$$
(S)
\left\{
\begin{array}{l}
\displaystyle\frac{dx}{dt}(t)\,=\, f(t , x(t)) + \sum_{j=1}^{m}g_{j}(t,x(t))\, \frac{du^j}{dt}(t) 
\quad \mbox{ a.e. } t \in [t_1,t_2] ,
\\
\displaystyle\frac{du}{dt}(t) \in \mathcal{C}\,\, \mbox{ a.e. } t \in [t_1,t_2],  
\end{array}
\right.
$$
\begin{itemize} 
\item[{\rm (i)}] $(S)$
is said to be {\rm quick $1$-controllable} w.r.t. the target set $\T$  at  a point  $(t_1,x_1,t_2,x_2) \in \T$ if, for any  covector  $\zeta=(\zeta_{t_1},\zeta_{x_1},\zeta_{t_2},\zeta_{x_2}) \in N_{\T}(t_1,x_1,t_2,x_2)$ such that $\zeta_{x_{2}} \not= 0$, we have
 \begin{equation} \label{QC}
\displaystyle \inf_{w\in \mathcal{C}} \,\zeta_{x_2}\cdot \,\sum_{j=1}^{m}  g_{j}(t_2, x_2)\,   w^j < 0 .
 \end{equation} 
 \item[{\rm (ii)}] 
 $(S)$ is said to be  {\rm drift-controllable} w.r.t. the target set $\T$ at  a point  $(t_1,x_1,t_2,x_2) \in \T$ if, for any  covector  $\zeta=(\zeta_{t_1},\zeta_{x_1},\zeta_{t_2},\zeta_{x_2}) \in N_{\T}(t_1,x_1,t_2,x_2)$ such that $\zeta_{x_{2}} \not= 0$, we have
 \begin{equation} \label{DC}
\displaystyle  \zeta_{x_2} \,\cdot\,f(t_2, x_2)  < 0 .
 \end{equation}
\end{itemize}
\end{definition}

\vsm 
\begin{proposition}\label{NE} 
 Consider the optimal control problem $(P)$ and its extended sense formulation $(P_{e})$. Let the data  satisfy hypothesis
{\rm (H1)}. Assume that there exists an extended sense minimizer $(\bar S,\bar y^0 ,\bar  y , \bar\nu, \bar \varphi^0 ,\bar \varphi )$ such that
 \begin{itemize}
 \item[{\rm (i)}] $\bar y^0(\bar S) > \bar y^0(0) $ (that is, the $t$-time interval is non-degenerate),
\item[{\rm (ii)}]  $\bar\nu(\bar S)<K$,
\item[{\rm (iii)}] $(S)$ is quick $1$-controllable w.r.t. ${\cal T}$ at $(\bar y^{0}(0),\bar  y(0), \bar y^{0}(\bar S), \bar y(\bar S))$.
 \end{itemize}
Then $(\bar S,\bar y^0 ,\bar  y ,\bar\nu, \bar \varphi^0 ,\bar \varphi )$ is a normal extremal and, in consequence of Theorem \ref{cor},  
there is no infimum gap.
\end{proposition}

\noindent
{\it Remark: } 
 The role of 1-quick controllability  as a no infimum gap sufficient condition  was earlier identified in  \cite{AMR}, 
in the case when ${\cal T}$ takes the form
$$
{\cal T}= \{0\}\times \{x_{0}\}\times \{T\}\times \hat {\cal T}\,,
$$
for some closed set $\hat {\cal T} \subset \R^{n}$ and some $T>0$.  Prop. \ref{NE} broadens the applicability of the earlier sufficient condition by no longer requiring ${\cal T}$ to have this special structure. 
 
\begin{proof}[Proof of Proposition \ref{NE}] The proof involves showing that the given extended sense minimizer $(\bar S,\bar y^0  , \bar y , \bar\nu, {\bar \varphi}^0 ,{\bar \varphi} )$
is a normal extremal; it will follow immediately from Thm.  \ref{cor} that there is no infimum gap.
 Suppose that $(\bar S,\bar y^0  , \bar y , \bar\nu, {\bar \varphi}^0 ,{\bar \varphi} )$ is not a normal extremal. Since, by assumption, $\bar\nu(\bar S)< K$ and  $\bar y^0(\bar S)-\bar y^0(0)>0$, we can deduce from  Thm.   \ref{PMPe}, there exists a set of multipliers $(p_0 , p ,\pi, \lambda)$, with $\lambda=0$,  $p \ne0$ and  $\pi=0$.  From  \eqref{fe3} and  \eqref{fe4} in  Thm.  \ref{PMPe} it may be deduced that
 $$
(p_0(0), p(0), p_0(\bar S), p(\bar S))=(\zeta_{t_1},\zeta_{x_1},-\zeta_{t_2},- \zeta_{x_2})
$$
for some $(\zeta_{t_1},\zeta_{x_1},-\zeta_{t_2},- \zeta_{x_2}) \in    N_{\T}(\bar y^0 (0),\bar y (0), \bar y^0 (\bar S), \bar y (\bar S))$ with $\zeta_{x_2}\ne0$.

\noindent From the almost everywhere condition  \eqref{fe3} and the continuity of $(p_0 ,p )$ we deduce that
\begin{equation}\label{f31}
\displaystyle 
  \max_{(w^0,w)\in \CC} \bigg\{-\zeta_{x_2}\cdot \big(f( \bar y^0 (\bar S),\bar y(\bar S))w^0  +  \sum_{j=1}^m g_{j}( \bar y^0 (\bar S),\bar y(\bar S))w^{j}\big)
-\zeta_{t_2}  \,w^0\bigg\}=0 
\end{equation}
and, choosing $w_0=0$,  we arrive at
\begin{equation}\label{contr}
\displaystyle   \min_{w\in \C\cap\partial\B_m} \,  \zeta_{x_2}\cdot  \sum_{j=1}^m g_{j}(\bar y^0(\bar S),\bar y(\bar S))w^{j}  
\;\ge \;0.
\end{equation}
This trivially violates the quick 1-controllability hypothesis. So $(\bar S,\bar y^0  , \bar y , \bar\nu, {\bar \varphi}^0 ,{\bar \varphi} )$ is a normal extremal.
\end{proof}
\vspace{0.1 in}

  If we assume ${\cal T}$ has an epigraph structure, as made precise below, then the assertions of Prop. \ref{NE} remain valid, either when it is  no longer assumed that the $t$-time interval is non-degenerate (condition (i)), or when we replace (i)--(iii)  by `slow controllability'.
   
  \begin{proposition}\label{slow}  
 Consider the optimal control problem $(P)$ and its extended sense formulation $(P_{e})$. Let the data  satisfy hypothesis
{\rm (H1)}. Assume that there exists an extended sense minimizer $(\bar S,\bar y^0 ,\bar  y , \bar\nu, \bar \varphi^0 ,\bar \varphi )$ such that
\begin{itemize}
\item[(i)] ${\cal T}$ is an epigraph set on a  neighborhood of $e:=(\bar y^{0}(0),\bar  y(0), \bar y^{0}(\bar S), \bar y(\bar S))$, in the sense that there exist  $\epsilon >0$ and a lower semicontinuous function $\psi :\R \times \R^{n}\times \R^{n}\to\R$ such that
$$
{\cal T}\cap (e +\epsilon \B_{1+2n})\,=\,\{(t_{1},x_{1},t_{2},x_{2})\,|\,
t_{2} \geq \psi(t_{1},x_{1},x_{2}) \}\cap (e +\epsilon \B_{1+2n})
$$
\item[(ii)] either $(S)$ is drift-controllable  w.r.t. ${\cal T}$ at $(\bar y^{0}(0),\bar  y(0), \bar y^{0}(\bar S), \bar y(\bar S))$ or $\bar\nu(\bar S)<K$ and system (S) is   quick $1$-controllable w.r.t. ${\cal T}$ at $(\bar y^{0}(0),\bar  y(0), \bar y^{0}(\bar S), \bar y(\bar S))$.
\end{itemize}
 Then $(\bar S,\bar y^0 ,\bar  y , \bar\nu, \bar \varphi^0 ,\bar \varphi )$ is a normal {extended sense} extremal and, in consequence of   Thm.  \ref{cor},  
there is no infimum gap.
\end{proposition}

 \begin{proof}  Once again, the proof involves showing that the given extended sense minimizer $(\bar S,\bar y^0  $, $\bar y , \bar\nu, {\bar \varphi}^0 ,{\bar \varphi} )$
is a normal extremal; the fact that there is no  infimum gap will then follow from  Thm.  \ref{cor}. 
 Suppose in contradiction that $(\bar S,\bar y^0  , \bar y , \bar\nu, {\bar \varphi}^0 ,{\bar \varphi} )$ is not a normal extremal.   {To start with}, notice that, since, by assumption, ${\cal T}$ is (locally) an epigraph set, we  can deduce from   Thm.   \ref{PMPe} that there exists a set of multipliers $(p_0 , p ,\pi, \lambda)$, with $\lambda=0$ and  $p \ne0$.  Indeed, if on the contrary $(p ,\lambda)=(0,0)$, then \eqref{fe1} in  Thm.  \ref{PMPe} implies that $p_0 \ne0$. By   the first equation in \eqref{fe2} we derive that $p_0 \equiv \bar p_0$ is constant,    from  \eqref{fe3}  it may be deduced that 
\bel{peg}
\bar p_0\le0,
\eeq
 while   \eqref{fe4}  implies that  
 $
(p_0(0), p(0), p_0(\bar S), p(\bar S))=(\zeta_{t_1},\zeta_{x_1},-\zeta_{t_2},- \zeta_{x_2})
$
for some  \linebreak $(\zeta_{t_1},\zeta_{x_1},-\zeta_{t_2},- \zeta_{x_2}) \in    N_{\T}(\bar y^0 (0),\bar y (0), \bar y^0 (\bar S), \bar y (\bar S))$. Now, since ${\cal T}$ is (locally) an epigraph set, we know that $\zeta_{_{t_{2}}} \leq 0$. Hence  $\bar p_0\ge0$, which together with \eqref{peg} implies that $\bar p_0=0$,  in contrast with  the hypothesis  $p_0 \ne0$. Thus there exists a set of multipliers  $(p_0 , p ,\pi, \lambda)$, with $\lambda=0$ and  $p \ne0$.

\noindent At this point, from   \eqref{fe3} and  \eqref{fe4} in  Thm.  \ref{PMPe} it may be deduced that
 $$
(p_0(0), p(0), p_0(\bar S), p(\bar S))=(\zeta_{t_1},\zeta_{x_1},-\zeta_{t_2},- \zeta_{x_2})
$$
for some $(\zeta_{t_1},\zeta_{x_1},-\zeta_{t_2},- \zeta_{x_2}) \in    N_{\T}(\bar y^0 (0),\bar y (0), \bar y^0 (\bar S), \bar y (\bar S))$ with $\zeta_{x_2}\ne0$.

\noindent From \eqref{fe3} and the continuity of $(p_0 ,p )$ we deduce that
\begin{equation}
 \label{f31}
\displaystyle 
  \max_{(w^0,w)\in \CC} \bigg\{-\zeta_{x_2}\cdot \big(f( \bar y^0 (\bar S),\bar y(\bar S))w^0  +  \sum_{j=1}^m g_{j}( \bar y^0 (\bar S),\bar y(\bar S))w^{j}\big)
-\zeta_{t_2}  \,w^0  +  \pi|w|  \bigg\}=0 \,.
\end{equation}
Now, if $(S)$ is drift-controllable  w.r.t. ${\cal T}$ at $(\bar y^{0}(0),\bar  y(0), \bar y^{0}(\bar S), \bar y(\bar S))$,  choosing $(w^{0},w)=(1,0)$,  we arrive at
 $-\zeta_{x_2}\cdot   f( \bar y^0 (\bar S),\bar y(\bar S)) - \zeta_{t_{2}} \leq 0$,
where $\zeta_{_{t_{2}}} \leq 0$.  {This yields} the contradiction to \eqref{DC}:
$$
\displaystyle   \, \zeta_{x_2}\cdot  f( \bar y^0 (\bar S),\bar y(\bar S)) \;\ge \;0.
$$

\noindent  In case $\bar\nu(\bar S)<K$ and system (S) is   quick $1$-controllable  w.r.t. ${\cal T}$ at $(\bar y^{0}(0),\bar  y(0), \bar y^{0}(\bar S), \bar y(\bar S))$, $\pi=0$  by  Thm.  \ref{PMPe}. {So,  } choosing $w_0=0$ in   \eqref{f31},  we {obtain} the following contradiction to  \eqref{QC}:
$$
\displaystyle   \min_{w\in \C\cap\partial\B_m} \,  \zeta_{x_2}\cdot  \sum_{j=1}^m g_{j}(\bar y^0(\bar S),\bar y(\bar S))w^{j}  \;\ge \;0.
$$
Hence $(\bar S,\bar y^0  , \bar y , \bar\nu, {\bar \varphi}^0 ,{\bar \varphi} )$ is a normal extremal in both cases.
 \end{proof}
 
 \vsm
The no infimum gap sufficient conditions derived up to this point arise from the properties of normal extremals provided by  Thm.  \ref{mainextended}. But normality-type conditions fail to cover some situations where we can demonstrate the absence of  an infimum gap by independent analysis. One such case is identified in the following proposition.
\begin{proposition}[The case without drift]
\label{no_drift}
Consider the optimal control problem $(P)$ and its extended sense formulation $(P_{e})$. Let the data  satisfy hypothesis
{\rm (H1)}. Assume that
$$
f  \equiv 0\ \quad\text{(`no drift').}
$$
Then there is no infimum gap.
\end{proposition}
\noindent
This sufficient condition is an immediate consequence of the lemma below, which  {yields directly the} following information about systems with no drift:  given an arbitrary extended sense process $(S, y^0  , y , \nu, \varphi^0 ,\varphi )$, we can find a strict sense process $(t_{1}, t_{2}, x , v, u )$ with the same endpoints, that is 
$$
y^{0}(0)=t_{1}, \ y(0)=  x(t_{1}),  \ 
y^{0}(S)=t_{2},  \  y(S)= x(t_{2}) \  \text{and} \   \nu( S)=v(t_{2}). 
$$
For such systems, an extended process is feasible if and only if  the corresponding strict sense process is feasible and the costs  (for problems $(P)$ and $(P_{e})$ respectively) are the same;  an infimum gap cannot then arise.
 
 \begin{lemma} Let hypothesis {\rm (H1)} be satisfied. Assume that
$$
f  \equiv 0\,.
$$
Then, given any extended sense process $(S,y^0 ,y , \nu, \varphi^0 ,\varphi )$ for \eqref{extended}, there exists a  strictly increasing, onto, absolutely continuous  map $\sigma : [t_1,t_2]\to [0,S]$ such that the trajectory-control pair 
\begin{equation}
\label{equiv}
(x , v , u)(t):= \Big(y, \nu, \varphi \Big)\circ\sigma(t) \quad \forall t\in [t_1,t_2]
\end{equation}
is  a strict sense process for \eqref{strict0}.
\end{lemma}

\begin{proof} Define 
$$ \hat\varphi^0(s) := t_1 +\frac{t_2-t_1}{S}\, s \quad \forall s\in[0,S].$$
Consider the bi-Lipschitz change of parameters $r:[0,S]\to [0,S']$
$$
r(s) := \int_0^s \Big(\frac{d\hat\varphi^0}{ds'}(s') +  \left|\frac{d\varphi}{ds'}(s')\right|\Big) ds' \quad\forall s\in[0,S]
$$
where $S' := r(S)$.
Write the inverse mapping  $s  := r^{-1} $. Now define the Lipschitz continuous functions
\begin{equation}
\label{strict1}
(\tilde\varphi^0,\tilde\varphi)(r):= (\hat\varphi^0,\varphi)\circ s(r) \quad\forall r\in [0,S'].
\end{equation}
We can show, by means of straightforward calculations, that 
\begin{itemize}
\item[(i)] $\ds \left(\frac{d\tilde\varphi^0}{dr}(r) ,\frac{d\tilde\varphi^0}{dr}(r)\right)\in \CC$ for a.e. $r\in [0,S']$;
\item[(ii)] $\tilde\varphi^0(0)=t_1$, $\tilde\varphi^0(S')=t_2$;
\item[(iii)] the path $(\tilde y^0,\tilde y, \tilde\nu) $ defined by    \begin{equation}\label{strict2}(\tilde y^0,\tilde y, \tilde\nu)(r):= \big(\tilde\varphi^0(r), y\circ s(r), \nu\circ s(r)\big)\quad\forall r\in [0,S']\end{equation}
coincides with the unique solution to \eqref{extended} corresponding to $(\tilde\varphi^0,\tilde\varphi) $ and initial state $(y^{0},y,\nu)(0)$.
\end{itemize}
(Note that, to establish (iii), we make use of our assumption that $f  \equiv 0$).
Since  $\ds \frac{d\tilde\varphi^0}{dr}(r) >0$ for a.e. $r\in [0,S']$, the extended sense process $(S,\tilde y^0,\tilde y, \tilde \nu, \tilde\varphi^0,\tilde\varphi)$ is an embedded strict sense process, i.e., $(x,v,u) $ defined by
$$
(x,v,u)(t):= (\tilde y, \tilde \nu,\tilde\varphi)\circ ({\tilde\varphi^0})^{-1}(t)\quad\forall t\in [t_1,t_2]
$$
is a strict sense process for \eqref{strict0} (in which $f  \equiv 0$). To complete the proof, we observe that \eqref{equiv} follows from \eqref{strict1} and \eqref{strict2}, when we choose $\sigma:= s\circ ({\tilde\varphi^0})^{-1}$.
\end{proof}
 \section{Examples}
We provide, in this section, a number of examples to illustrate the preceding theory. 
 
\begin{example}\label{es1} {\rm This example tells us that an infimum gap can actually occur when the sufficient condition of  Thm.  \ref{cor} is violated. 
$$
\left\{
\begin{array}{l}
\mbox{Minimize } -x_{1}(1)
\\ [1.5ex]
\mbox{subject to}
\\ [1.5ex]
\ds\frac{d x_{1}}{dt}(t) = \frac{d u}{dt}(t) \  \mbox{a.e. }t\in [0,1] 
\\ [1.5ex]
\ds\frac{d x_{2}}{dt}(t)= x_{1}(t)  \ \mbox{a.e. }t\in [0,1]
\\ [1.5ex]
\ds\frac{d v}{dt}(t) = \left|\frac{d u}{dt}(t)\right| \  \mbox{a.e. }t\in [0,1] 
\\ [1.5ex]
\ds\frac{d u}{dt}(t)\geq 0  \ \mbox{a.e. }t\in [0,1]
\\ [1.5ex]
v(0)=0, \ v(1)\le 1, \ x_{1}(0)=x_{2}(0)=0 \mbox{ and }  x_{2}(1)\leq0.
\end{array}
\right.
$$
In this special case of  $(P)$, $n=2$, $m=1$,  ${\cal T}= \{0\} \times \{0_2\} \times \{1\} \times  
\hat{{\cal T}}$, 
in which 
$\hat{{\cal T}}= \R \times (-\infty,0]$, ${\cal C} = [0, +\infty)$, $K=1$ and
$$
f(x)= \left( 
\begin{array}{c}
0
\\
x_{1}
\end{array}
\right),\ \
g_{1}=
\left( 
\begin{array}{c}
1
\\
0
\end{array}
\right) . 
$$
The extended problem is
$$
\left\{
\begin{array}{l}
\mbox{Minimize } -y_1(S)
\\  [1.5ex]
\mbox{over $S>0$, } (y^0,y, \nu, \varphi^0,\varphi) \in W^{1,1}([0,S];\R\times\R^{2}\times\R\times\R\times\R )   \quad
\mbox{satisfying }\\  [1.5ex]
\ds \frac{d y^{0}}{ds}(s)= \frac{d  \varphi^0}{ds}(s) \mbox{ a.e. }s \in [0,S]
\\ [1.5ex]
\ds\frac{d  y^{1}}{ds}(s) =  \frac{d  \varphi}{ds}(s) \ \ \  \mbox{ a.e. }s \in [0,S]\,,
\\ [1.5ex]
\ds\frac{d  y^{2}}{ds}(s)=  y^{1}(s)\,\frac{d  \varphi^0}{ds}(s) \ \ \mbox{ a.e. }s \in [0,S]\,,
\\ [1.5ex]
\ds\frac{d  \nu}{ds}(s)=   \,\left|\frac{d  \varphi}{ds}(s)\right| \ \ \mbox{ a.e. }s \in [0,S]\,,
\\ [1.5ex]
\ds \bigg(\frac{d  \varphi}{ds}(s), \frac{d  \varphi}{ds}(s)\bigg) \in \{(w_0,w)\in \R_+\times\R_+: \ w_0+w=1\}  \ \  \mbox{ a.e. }s \in [0,S]\,,
\\ [1.5ex]
 \nu(0)=0, \ \nu(S)\le 1, \ (y^{0}(0),y^{1}(0),y^{2}(0))=(0,0,0), \quad y^{0}(S)=1, \quad    y^2(S)\le0. 
\end{array}
\right.
$$
It is straightforward to show that $(\bar S,\bar y^{0} , \bar y , \bar\nu,  \bar \varphi^{0} ,\bar  \varphi )$
given by $\bar S =2$, $\ds\frac{d \bar \varphi^{0}}{ds} \equiv 1-\bar w $,  $\ds\frac{d \bar \varphi}{ds} \equiv \bar w  $, in which  $\bar y =( \bar y^1, \bar y^2) $, $\bar y_{2} \equiv 0$, $\bar \nu\equiv \bar y_{1}$ 
and
\begin{eqnarray*}
&& \bar y_{1}(s)=
\left\{
\begin{array}{ll}
0& \mbox{for } 0\leq s\leq 1
\\
s-1& \mbox{for } 1< s\leq 2
\end{array}
\right.
,\,
\bar y^{0}(s)=
\left\{
\begin{array}{ll}
s& \mbox{for } 0\leq s\leq 1
\\
1& \mbox{for } 1< s\leq 2
\end{array}
\right.,\,
\bar w(s)=
\left\{
\begin{array}{ll}
0& \mbox{for } 0\leq s\leq 1
\\
1& \mbox{for } 1< s\leq 2
\end{array}
\right. 
\end{eqnarray*}
is a minimizer for the extended problem. Furthermore the infimum costs for the original and extended problems are, respectively,
$$
\inf (P) =0 \mbox{ and } \inf (P_{e}) =- 1\,.
$$
Since there is an infimum gap, all the the sufficient conditions of the previous section, for non-occurence of an infimum gap, must be violated. In connection with Prop. \ref{NE} we note that the vector $\zeta =(0,1)$ lies in $N_{\hat {{\cal T}}}(\bar y(\bar S))= \{0\}\times [0,\infty) $ and
$$
\zeta\cdot g_{1}\, w \geq 0 
$$
for every $w \in {\cal C}$, in violation of the
 the quick 1-controllability  condition. (Also, condition  $\bar\nu(\bar S)< K$ is violated).   Concerning the Prop. \ref{slow}, we see that
$$
\zeta \cdot f(\bar y(\bar S)) =0\,
$$
in violation of the drift controllability condition.  Consider the normality condition. We can establish by simple calculations that $(\bar S,\bar y^{0} , \bar y , \bar\nu,  \bar \varphi^{0} ,\bar  \varphi )$ is an extremal, and a possible Lagrange multiplier set is
$(p_0 ,  p_{1} , p_{2} , \pi, \lambda)$, in which
$$
p_0 \equiv  0,  \ \ p_{1}  \equiv 0,\ \  p_{2}  \equiv c,  \ \   \pi=0, \ \lambda=0, 
$$
for any constant $c >0$. Notice that $\lambda=0$, so $(\bar S,\bar y^{0} , \bar y , \bar\nu,  \bar \varphi^{0} ,\bar  \varphi )$ is an abnormal extremal. Since there is a unique minimizer\footnote{ Of course, `unique' here means `up to {translations} of $(\varphi^0,\varphi)$'.} for the extended problem, we have shown that all minimizers for the extended problem are abnormal. Thus, the sufficient condition of  Thm.  \ref{cor} is violated.
}
\end{example}

\begin{example}\label{es2} {\rm  This example aims to demonstrate that the `no infimum gap' sufficient condition of  Thm. \ref{cor}, based on normality, is distinct from those of Props. \ref{NE} and \ref{slow}, based on quick $1$-controllability and on drift controllability respectively. 
%
%
Consider the problem
\bel{strictex3}
\left\{
\begin{array}{l}
\mbox{Minimize } h(x(1))
\\ [1.5ex]
\mbox{over } (x,v, u) \in W^{1,1}([0,1];\R^{3}\times\R\times\R^2)   \quad
\mbox{satisfying }\\ [1.5ex]
\begin{array}{l}
\ds \frac{d x}{dt}(t) = f(x(t))+g_1(x(t)) \, \frac{d u_1}{dt}(t) + g_2(x(t))  \, \frac{d u_2}{dt}(t)  \,, \quad \mbox{ a.e. } t \in [0,1] ,
\\ [1.5ex]
\ds \frac{d v}{dt}(t) =  \, \left|\frac{d u}{dt}(t)\right|  \,, \quad \mbox{ a.e. } t \in [0,1] ,
\\ [1.5ex]
\displaystyle\frac{du}{dt}(t)\in\C:=\R^2  \  \mbox{ a.e. } t \in [0,1],   \\ [1.5ex]
\displaystyle v(0)=0, \ \ v(1)   \leq K ,\quad 
 x(0)=(1,0,0), \quad   x(1)\in\hat{\mathcal{T}} .\end{array}
\end{array}
\right.
\eeq
in which $n=3$, $m=2$, K=2, $h(x):= -x_{1}$, $\hat{\mathcal{T}}:= \{(x_1 , x_2,x_3) \,|\, \  x_1\le0, \ x_2 \leq 0, \   x_3 \leq 0\}$
\bel{gnhi}
\ds g_1(x):=\left(\begin{array}{l} \ 1 \\ \ 0\\x_2\end{array}\right),   \quad \ds g_2(x):=\left(\begin{array}{l} 0 \\ 1\\-x_1\end{array}\right) , \quad  \ds f(x):=\left(\begin{array}{l} \ 0 \\  x_2\\ \ 0\end{array}\right), \quad \forall  x\in \R^3\,.
\eeq
This is an example of problem $(P)$, in which the underlying control system is a modification of the nonholonomic integrator, to include  {a} non-zero drift term.
\ \\

\noindent
The extended  problem is
\bel{ex31}
\left\{
\begin{array}{l}
\mbox{Minimize } h(y(S))
\\ [1.5ex]
\mbox{over $S>0$, } (y^0,y,\nu, \varphi^0,\varphi) \in W^{1,1}([0,S];\R\times\R^{3}\times\R\times\R\times\R^2)   \quad
\mbox{satisfying }\\ [1.5ex]
\begin{array}{l}
\displaystyle \frac{dy^0}{ds}(s)\,= \,\frac{d\varphi^0}{ds}(s)  \,, \quad \mbox{ a.e. } s \in [0,S] , 
 \\ [1.5ex]
\displaystyle \frac{dy}{ds}(s)\,= f(y(s)) \,\frac{d\varphi^0}{ds}(s)+ g_1(y(s)) \, \frac{d \varphi^1}{ds}(s) + g_2(y(s))  \, \frac{d \varphi^2}{ds}(s)  \,, \quad \mbox{ a.e. } s \in [0,S] , 
\\ [1.5ex]
\displaystyle \frac{d\nu}{ds}(s)\,= \,\left|\frac{d\varphi }{ds}(s)\right|  \,, \quad \mbox{ a.e. } s \in [0,S] , 
 \\ [1.5ex]
\displaystyle \left(\frac{d\varphi^0}{ds}(s), \frac{d\varphi}{ds}(s)\right)\in\CC:=\left\{(w_0,w)\in\R_+\times\R^2:   \ w_0+|w|=1\right\}  \,\, \mbox{ a.e. } s \in [0,S],   \\ [1.5ex]
\nu(0)=0, \ \ \nu(S)  \leq K ,\quad 
\displaystyle y(0)=(1,0,0), \quad    y^0(0)=0, \quad y^0(S)=1, \quad  y (S)\in\hat{\mathcal{T} }.\end{array}
\end{array}
\right.
\eeq
It is straighforward to show that 
the feasible extended sense process  $(\bar S, \bar y^0 , \bar y , \bar\nu,  \bar \varphi^0 , \bar \varphi )$, where $\bar S=2$, 
\bel{optc}
\bigg(\frac{d\bar \varphi^0}{ds},\frac{d\bar \varphi}{ds}\bigg) =\bigg(\frac{d\bar \varphi^0}{ds} ,\frac{d\bar \varphi^1}{ds} ,\frac{d\bar \varphi^2}{ds}\bigg) =(1,0,0)\chi_{[0,1]} +(0,-1,0)\chi_{[1,2]} \,,
\eeq
  and  
\bel{optt}
(\bar y^0,\bar y,\bar\nu) =(\bar y^0,\bar y^1,\bar y^2,\bar y^3,\bar\nu)  =(s,1,0,0,0 )\chi_{[0,1]} +(1,2-s,0,0,s-1)\chi_{[1,2]} \,,
\eeq
is the  unique minimizer. By  Thm.  \ref{PMPe} there exists a set of multipliers $(p ,p_0,\pi, \lambda)$ with $(p ,\lambda)\ne0$, where  $\lambda\ge0$, $p_0\in\R$  and  $\pi=0$,  since $\frac{\partial h}{\partial v}\equiv0$ and  $\bar\nu(\bar S)=1<2$. The costate trajectory $p $ satisfies the differential system
\bel{ex3f1}
\left\{
\begin{array}{l}
\displaystyle \frac{dp_1}{ds}(s)\,= 0, \\ \, \\
\displaystyle \frac{dp_2}{ds}(s)\,=-p_2(s)\chi_{[0,1]}(s) +p_3(s)\chi_{[1,2]}(s)\,  \,, \quad \mbox{ a.e. } s \in [0,2] ,  \\ \, \\
\displaystyle \frac{dp_3}{ds}(s)\,=0\,.
\end{array}\right.
\eeq
We see that  $p_1 \equiv \bar p_1$ and $p_3 \equiv \bar p_3$ are constants.  The triple $(p ,p_0, \lambda)$   verifies the transversality condition
\bel{ex3f2}
-(p_1,p_2,p_3)(2)\in\lambda(-1,0,0)+N_{\hat{\mathcal{T}}}(0,0,0), 
\eeq
and, for a.e. $s\in[0,2]$,  the relations
\bel{ex3f3}
\begin{array}{l}
 p_0\chi_{[0,1]}(s)-\bar p_1\chi_{[1,2]}(s)= \\ 
 \max_{(w_0,w)\in\CC}\bigg\{p_0\,w_0+\bar  p_1\,w_1+p_2(s)\,( \bar y_2(s)\,w_0+w_2)+\bar p_3\,\big(\bar y_2(s)\,w_1-\bar y_1(s)\,w_2\big)\bigg\}=0.
 \end{array}
\eeq
From the first relation in  \eqref{ex3f3} we deduce that $p_0=\bar p_1=0$.  Then, since $\bar y_2 \equiv 0$ and $(w_0,w_1,w_2)=(0,0,\pm1)\in\CC$,  by the second relation in \eqref{ex3f3}  we necessarily have 
\bel{ex3p2}
p_2(s)= \bar p_3\bar y_1(s) \quad \forall   s\in[0,2].
\eeq 
Notice that $N_{\hat{\mathcal{T}}}(0,0,0)=\{(\alpha,\beta,\gamma)\}$ for any $\alpha$, $\beta$ and $\gamma\ge0$. Hence condition \eqref{ex3f2} implies 
$$
-(p_1,p_2,p_3)(2)=(-\lambda+\alpha, \beta, \gamma)
$$
and we arrive to the equalities 
$$
\lambda=\alpha, \quad \bar p_3=-\gamma, \quad  p_2 =(\gamma-\beta)\,e^{1-s}\,\chi_{[0,1]} +\big(\gamma(2-s)-\beta\big)\,\chi_{[1,2]} ,
$$
 It follows simply from \eqref{ex3p2} that $\gamma=\beta=0$, so that $\bar p_3=0$ and $ p_2 \equiv 0$. This proves that $(\bar S, \bar y^0, \bar y, \bar\nu, \bar \varphi^0, \bar \varphi) $ is a normal extremal,  since $(p , \lambda)\ne0$ if and only if $\lambda=\alpha>0$. We have shown that the conditions of  Thm.   \ref{cor} are satisfied and, in consequence, there  is no infimum gap.
 }
 \end{example}
 
 It is a straightforward matter to check that the quick $1$-controllability and drift controllability conditions in Props. \ref{NE} and \ref{slow}, respectively, are both violated. We know from Prop. \ref{NE} and Prop. \ref{slow}, that the sufficient condition of  Thm.  \ref{cor}  covers all cases when Props. \ref{NE} and \ref{slow}  exclude infimum gaps. This example  goes further, by showing that, in some cases, the sufficient condition of  Thm.  \ref{cor} excludes an infimum gap, when the other two conditions fail to do so and therefore has broader potential application.  
 
\begin{example}\label{es3} {\rm  The purpose of this example is to demonstrate that  the `no infimum gap' sufficient condition of  Thm.  \ref{cor}, based on normality,  is not necessary.  Consider the optimal control problem 
 \bel{strictex2}
\left\{
\begin{array}{l}
\mbox{Minimize } h(x(1))
\\ [1.5ex]
\mbox{over } (x,v, u) \in W^{1,1}([0,1];\R^{3}\times\R\times\R^2)   \quad
\mbox{satisfying }\\ 
[1.5ex]
\begin{array}{l}
\ds \frac{d x}{dt}(t) = g_1(x(t)) \, \frac{d u_1}{dt}(t) + g_2(x(t))  \, \frac{d u_2}{dt}(t)  \,, \quad \mbox{ a.e. } t \in [0,1] ,
\\ [1.5ex]
\ds \frac{d v}{dt}(t) =  \,\left| \frac{d u }{dt}(t)\right|  \,, \quad \mbox{ a.e. } t \in [0,1] ,
\\ [1.5ex]
\displaystyle\frac{du}{dt}(t)\in\C:=\{(w_1,w_2)\,|\, \  w_1\in\R, \ w_2\ge0\}  \  \mbox{ a.e. } t \in [0,1],   \\ [1.5ex]
\displaystyle v(0)=0, \ \ v(1)   \leq K ,\quad 
 x(0)=(1,0,0), \quad   x(1)\in\hat{\mathcal{T}},\end{array}
\end{array}
\right.
\eeq
in which $n=3$, $m=2$, $K=2$, $h $ and $g_{1} $, $g_{2} $  and $\hat{{\cal T}}$ are as in the previous example. Note however that, now, the drift term $f \equiv 0$.
The extended  problem is
\bel{ex21}
\left\{
\begin{array}{l}
\mbox{Minimize } h(y(S))
\\ [1.5ex]
\mbox{over $S>0$, } (y^0,y, \nu, \varphi^0,\varphi) \in W^{1,1}([0,S];\R\times\R^{3}\times\R\times\R\times\R^2)   \quad
\mbox{satisfying }\\
[1.5ex]
\begin{array}{l}
\displaystyle \frac{dy^0}{ds}(s)\,= \,\frac{d\varphi^0}{ds}(s)\,, \quad \mbox{ a.e. } s \in [0,S] ,  \\ [1.5ex]
\displaystyle \frac{dy}{ds}(s)\,= g_1(y(s)) \, \frac{d \varphi^1}{ds}(s) + g_2(y(s))  \, \frac{d \varphi^2}{ds}(s)  \,, \quad \mbox{ a.e. } s \in [0,S] , 
\\ [1.5ex]
\displaystyle \frac{d\nu}{ds}(s)\,= \,\left|\frac{d\varphi}{ds}(s)\right| \,, \quad \mbox{ a.e. } s \in [0,S] ,  \\ [1.5ex]
\displaystyle \left(\frac{d\varphi^0}{ds}(s), \frac{d\varphi}{ds}(s)\right)\in\CC:=\left\{(w_0,w)\in\R_+\times\C:   \ w_0+|w|=1\right\}  \,\, \mbox{ a.e. } s \in [0,S],   \\ [1.5ex]
\nu(0)=0, \ \ \nu(S)   \leq 2 ,\quad 
\displaystyle y(0)=(1,0,0), \quad    y^0(0)=0, \quad y^0(S)=1, \quad  (y_1,y_2,y_3)(S)\in\hat{\mathcal{T} }.\end{array}
\end{array}
\right.
\eeq
The minimizing extended sense process  $(\bar S, \bar y^0 , \bar y , \bar\nu, \bar \varphi^0 , \bar \varphi )$ for the optimal control problem (\ref{strictex2}), studied in the  previous example, given by
\eqref{optc}, \eqref{optt}, is a minimizing extended sense process also for problem \eqref{ex21}. The multiplier set $(\lambda,p ,p_0,\pi)$, where  $(\lambda,p )\ne0$,  $\lambda\ge0$, $p_0\in\R$,   is such that  $\pi=0$,  since   $\frac{\partial h}{\partial v}\equiv0$ and  $\bar\nu(\bar S)=1<2$. The costate trajectory $p $ satisfies
\bel{ex2f1}
\left\{
\begin{array}{l}
\displaystyle \frac{dp_1}{ds}(s)\,= 0, \\ \, \\
\displaystyle \frac{dp_2}{ds}(s)\,= p_3(s)\chi_{[1,2]}(s)\,  \,, \quad \mbox{ a.e. } s \in [0,2] ,  \\ \, \\
\displaystyle \frac{dp_3}{ds}(s)\,=0
\end{array}\right.
\eeq
We see that $p_1 \equiv \bar p_1$ and $p_3 \equiv \bar p_3$ are constants.  We also know that,   for a.e. $s\in[0,2]$,
\bel{ex2f3}
 p_0\chi_{[0,1]}(s)-\bar p_1\chi_{[1,2]}(s)=\max_{(w_0,w)\in\CC}\bigg\{ p_0w_0+\bar p_1\,w_1+p_2(s)\, w_2+\bar p_3\,\big(\bar y_2(s)\,w_1-\bar y_1(s)\,w_2\big)\bigg\}=0.
\eeq
From   \eqref{ex2f3} we deduce that $p_0=\bar p_1=0$ and,  since $\bar y_2 \equiv 0$,  $\bar y_1 =\chi_{[0,1]} +(2-s)\,\chi_{[1,2]} $ and $w_2\ge0$ for all $(w_0,w)\in\CC$,  also that
\bel{ex2p2}
p_2(s)\le  \bar p_3\quad \forall   s\in[0,1] , \qquad  p_2(s)\le \bar p_3\,(2-s) \quad \forall   s\in[1,2].
\eeq 
We deduce from the transversality condition, as in the previous example, that
$$
-(p_1,p_2,p_3)(2)=(-\lambda+\alpha, \beta, \gamma)
$$
which implies
$
\lambda=\alpha$ $\bar p_3=-\gamma$ and 
\begin{equation}
\label{multiplierp2} 
 p_2 =(\gamma-\beta)\,\chi_{[0,1]} +\big(\gamma(2-s)-\beta\big)\,\chi_{[1,2]} .
\end{equation}
 \eqref{ex2p2}  implies $\beta\ge2\gamma$. 
 Choosing $\alpha =0$,  and any $\gamma >0$, $\beta > 2 \gamma$, we arrive at a multiplier set $(p_{0}=0, p_{1}\equiv 0, p_{2} , p_{3} \equiv- \gamma,  \pi =0,  \lambda =0)$, in which $p_{2} $ is given by \eqref{multiplierp2}. We have shown that the unique minimizer is not a normal extremal, so the sufficient condition for no infimum gap of  Thm.  \ref{cor} is not applicable to this problem.
 
 Recall that the optimal control problem of this example has no drift. For such problems, we know that there can be no infimum gap. (See Prop. \ref{no_drift}).  Consequently, the sufficient condition of  Thm.  \ref{cor} fails to eliminate the possible occurrence of an infimum gap, in some circumstances when we can establish, by other means, that there is no infimum gap.
 }
 \end{example}


\section*{References} 

 \begin{thebibliography}{00}

\bibitem{AR} M.S. Aronna, F. Rampazzo,  {\it $L^1$ limit solutions for control systems.} J. Differential Equations 258, no. 3, 954--979, 2015. 

\bibitem{AMR} M.S. Aronna, M. Motta, F. Rampazzo,  {\it Infimum gaps for limit solutions.} Set-Valued Var. Anal. 23, no. 1, 3--22, 2015.
         
 \bibitem{AKP1}   A. Arutyunov, D. Karamzin and  F. Pereira,     {\it Pontryagin's maximum principle for constrained impulsive control problems.} Nonlinear Anal. 75, no. 3, 1045�1057,  2012.
 
  \bibitem{AKP3}   A. Arutyunov,  D. Karamzin and  F. Pereira,     {\it   A nondegenerate maximum principle for the impulse control problem with state constraints.} SIAM J. Control Optim. 43, no. 5, 1812�1843,  2005.
   
\bibitem {AB} D. Azimov,  R. Bishop (2005) {\it  New trends in astrodynamics and applications: optimal trajectories for space
guidance.}  Ann. New York Acad. Sci. 1065(1), 189--209. 

\bibitem{bonnans} 
	 	{Bonnans, J. F. and Shapiro, A.}
	  	\emph{Perturbation Analysis of Optimization Problems},
	  	{Springer, New York},
		{2000}
\bibitem{BR} A.Bressan,  F. Rampazzo,  {\it
 On differential systems with vector-valued impulsive controls.} Boll. Un. Mat. Ital. B (7) 2, no. 3, 641--656, 1988.
 
 \bibitem{BP} A. Bressan, B. Piccoli, {\it Introduction to the mathematical theory of control.}
AIMS Series on Applied Mathematics, 2. American Institute of Mathematical Sciences (AIMS), Springfield, MO, 2007. 
 
\bibitem{BR1} A.Bressan,  F. Rampazzo,  {\it   Moving constraints as stabilizing controls in classical mechanics.}  Arch. Ration. Mech. Anal. 196, no. 1, 97--141,  2010.
 
 \bibitem{AB1} Aldo Bressan, {\it Hyper-impulsive motions and controllizable coordinates for Lagrangean systems,} Atti Accad. Naz. Lincei, Memorie, Serie VIII, Vol. XIX, 197--246, 1990.
 
 \bibitem{AB2} Aldo Bressan, {\it On some control problems concerning the ski or swing,}  Atti Accad. Naz. Lincei, Memorie, Serie IX, Vol. I,
 147--196,  1991.
 
  \bibitem{CSWML}  A. Catll\'a, D. Schaeffer, T. Witelski,  E. Monson,  A. Lin, (2008)  {\it On spiking models for synaptic activity and
impulsive differential equations,} SIAM Rev. 50(3), 553--569 

\bibitem{ClarkeLed}   F. H. Clarke,
Y. S. Ledyaev,
R. J. Stern and P. R. Wolenski,
{\em Nonsmooth Analysis and Control Theory}\/,
Graduate Texts in Mathematics vol. 178, Springer-Verlag,
New York, 1998.  
 
 \bibitem{Dontchev}A.L. Dontchev, T. Zolezzi, {\it Well-posed optimization}  Springer-Verlag, New York, 1993.
 
   \bibitem {GRR} P. Gajardo, H. Ramirez C., A. Rapaport, {\it Minimal time sequential batch reactors with bounded and impulse controls
for one or more species, }SIAM J. Control Optim. 47 (6)  2827--2856,  2008.

 
 \bibitem{GS} M. Guerra,  A. Sarychev,  {\it Fr\'echet generalized trajectories and minimizers for variational problems of low coercivity,} J. Dyn. Control Syst.  21,  no. 3,  351--377, 2015.
 
 \bibitem{Haiek}   O. H\'ajek, {\it Discontinuous differential equations I,} J. Differential Equations 32, 149--170, 1979.

 \bibitem {KDPS}
 D.Y. Karamzin; V.A. de Oliveira,  F.L. Pereira, G.N. Silva,  {\it On the properness of an impulsive control extension of dynamic optimization problems,} ESAIM Control Optim. Calc. Var.  21,  no. 3, 857--875, 2015.
 
\bibitem{MiRu} M. Miller, E. Y. Rubinovich,  {\it Impulsive control in continuous and discrete-continuous systems.} Kluwer Academic/Plenum Publishers, New York, 2003.

  \bibitem{MR} M. Motta, F. Rampazzo,  {\it  Space-time trajectories of nonlinear systems driven by ordinary and impulsive controls.} Differential Integral Equations 8, no. 2, 269--288, 1995.
  
 \bibitem{MR1}  M. Motta, F. Rampazzo, {\it  Dynamic programming for nonlinear systems driven by ordinary and impulsive controls. } SIAM J. Control Optim. 34, no. 1, 199--225, 1996. 

 \bibitem{MR2}  M. Motta, F. Rampazzo, {\it State-constrained control problems with neither coercivity nor $L^1$ bounds on the controls.}  Ann. Mat. Pura Appl. (4) 177, 117--142,  1999.
 
  \bibitem{MS}  M. Motta, C. Sartori, {\it  On asymptotic exit-time control problems lacking coercivity. } ESAIM Control Optim. Calc. Var. 20, no. 4, 957--982, 2014. 
   
\bibitem{RW} R. T. Rockafellar and R. J.-B.
Wets,
{\em Variational Analysis}\/, Grundlehren der Mathematischen
Wissenschaften, 317, Springer-Verlag, New York, 1998.

 
\bibitem{SV} 
 \newblock G. Silva, R. Vinter,
 \newblock  {\it  Measure driven differential inclusions,}
   \newblock  J. Math. Anal. Appl.,  \textbf{202 no. 3},  727--746,  1996.
   
\bibitem{SV1} G. Silva, R. Vinter,  {\it   Necessary conditions for optimal impulsive control problems.} SIAM J. Control Optim. 35, no. 6, 1829--1846,  1997.

\bibitem{Vinter} R. B. Vinter, \emph{Optimal Control.} Birkh\"auser, Boston, 2000.

\bibitem{PV1} M. Palladino and R. B. Vinter, {\it Minimizers That Are Not Also Relaxed Minimizers}, SIAM J. Control and Optim. 52, 4, 2164--2179, 2014.
\bibitem{PV2} M. Palladino and R. B. Vinter, {\it When are Minimizing Controls also Minimizing Relaxed Controls?,}  Discrete and Continuous Dyn. Systems �  Series A, 35, 4573--4592, 2015.

\bibitem{Suss} H.J. Sussmann, {\it On the Gap Between Deterministic and Stochastic Ordinary Differential Equations}
The Annals of Probability, 
Vol. 6, No. 1,
 pp. 19-41
\bibitem{R} F. Rampazzo, {\it  On the Riemannian structure of a Lagrangian system and the problem of adding time-dependent constraints as controls.} European J. Mech. A Solids 10, no. 4, 405--431,  1991.

\bibitem{warga1} J. Warga, {\it Normal Control Problems have no Minimizing Strictly Original Solutions},
Bulletin of the Amer. Math. Soc., 77, 4, 625-628, 1971.

\bibitem{warga} J. Warga, {\it Optimal Control of Differential and Functional Equations}, Academic Press, New York, 1972.

\bibitem{WZ}  
 P. Wolenski, S. \v{Z}abi\'{c}, 
  {\it   A differential solution concept for impulsive systems,} Dyn.
Contin. Discrete Impuls. Syst. Ser. A Math. Anal., 13B,  199--210,  2006.
  
\end{thebibliography}
\end{document}